\documentclass[10pt]{amsart}

\title[Weak specification for geodesic flow on CAT(-1) spaces]{The weak specification property for geodesic flows on CAT(-1) spaces}

\author{David Constantine}
\address{
Wesleyan University \\
Mathematics and Computer Science Department \\
Middletown, CT 06459}
\email{dconstantine@wesleyan.edu}

\author{Jean-Fran\c{c}ois Lafont}
\address{Department of Mathematics\\
                 Ohio State University\\
                 Columbus, Ohio 43210}
\email{jlafont@math.ohio-state.edu}

\author{Daniel J. Thompson}
\address{Department of Mathematics\\
                 Ohio State University\\
                 Columbus, Ohio 43210}
\email{thompson.2455@osu.edu}

\date{\today}

\thanks{D.C. thanks the Ohio State University Math Department for hosting him for a semester during which much of this work was done.
J.-F.L. is supported by NSF grants DMS-1510640, DMS-1812028. D.T. is supported by NSF grant DMS-$1461163$.} 

\subjclass[2000]{37D35, 37D40, 37A20, 51F99} 

\keywords{Locally CAT(-1) space, geodesic flow, weak specification property, equilibrium measure, Gibbs property, measure of maximal entropy, large deviations property.}

\newtheorem{thm}{Theorem}[section]
\newtheorem{thmx}{Theorem}

\newtheorem{lem}[thm]{Lemma}
\newtheorem{prop}[thm]{Proposition}
\newtheorem{cor}[thm]{Corollary}

\theoremstyle{definition}

\newtheorem{defn}[thm]{Definition}

\numberwithin{equation}{section}

\def\Pb{\ifmmode{\Bbb P}\else{$\Bbb P$}\fi}
\def\Z{\ifmmode{\Bbb Z}\else{$\Bbb Z$}\fi}
\def\Q{\ifmmode{\Bbb Q}\else{$\Bbb Q$}\fi}
\def\C{\ifmmode{\Bbb C}\else{$\Bbb C$}\fi}
\def\R{\ifmmode{\Bbb R}\else{$\Bbb R$}\fi}
\def\H{\ifmmode{\Bbb H}\else{$\Bbb H$}\fi}
\def\Susp{\operatorname{Susp}}
\def\CAT{\operatorname{CAT}}
\def\Var{\operatorname{Var}}
\def\Per{\operatorname{Per}}

\def \EEE{\mathcal E}

\def \AAA{\mathcal A}
\def \FFF{\mathcal F}

\usepackage{graphicx} 
\usepackage{amsmath,amssymb,amsthm,amsfonts,amscd,flafter,epsf,bm}
\usepackage{mathtools}
\usepackage{graphics}
\usepackage{MnSymbol}
\usepackage{epsfig}
\usepackage{psfrag}
\usepackage{xcolor}
\usepackage{comment}

\input xy
\xyoption{all}

\begin{document}

\begin{abstract}

We prove that the geodesic flow on a compact locally $\CAT(-1)$ space has the weak specification property, and give various applications. We show that every H\"older potential on the space of geodesics has a unique equilibrium state. We establish the equidistribution of weighted periodic orbits and the large deviations principle for all such measures. The thermodynamic results are proved for the class of expansive flows with weak specification. 
\end{abstract}

\maketitle

\setcounter{secnumdepth}{2}

\setcounter{section}{0}

\section{Introduction}

An important characteristic of hyperbolic dynamical systems is the {\it specification property}, 
introduced by Bowen in the early 1970s. The geodesic flow of a negatively curved Riemannian manifold is a prime example of a flow 
satisfying the specification property. Bowen used the specification property to establish
a number of fundamental results about the ergodic properties of such geodesic flows (and more generally, for Axiom A flows), showing for example the equidistribution of prime closed geodesics to an ergodic measure of maximal entropy \cite{bowen-periodic}. These results were proved before Bowen established the existence of Markov partitions and associated symbolic dynamics for these geodesic flows \cite{rB73}. 
Beyond uniform hyperbolicity, the paradigm remains that while proofs of the stronger properties of hyperbolic dynamics require the system to be described by symbolic dynamics \cite{Bowen-Ruelle, PP}, an approach using the specification property affords greater flexibility, and still yields many interesting results. In this paper, we investigate the geodesic flow on locally $\CAT(-1)$ spaces, using geometric arguments to obtain a weak version of the specification property. Once we have the necessary dynamical properties of the flow from these geometric arguments, we proceed using purely analytic arguments to obtain many dynamical properties of the geodesic flow. 

The class of compact {\it locally $\CAT(-1)$ spaces} was popularized in the 1980s by Gromov, as a far reaching generalization of negatively curved Riemannian manifolds. 
To any such space $X$, one can associate the space $GX$ of all bi-infinite geodesics in $X$. The space $GX$ is a compact 
metric space, and possesses a natural $\mathbb R$-flow by shifting the parametrization of 
geodesics -- this is known as \emph{the geodesic flow} since it generalizes the geodesic flow on a Riemannian manifold. A natural problem is to develop Bowen's approach for 
this broader class of flows. Our first result is the following:

\begin{thmx}\label{thm:weak-specification}
Let $X$ be a compact, locally $\CAT(-1)$, geodesic metric space, with fundamental group not isomorphic to $\mathbb Z$. Then the geodesic flow on 
$GX$ satisfies the weak specification property. Furthermore, the geodesic flow is expansive and any H\"older continuous function $\varphi: GX \to \mathbb R$ has the Bowen regularity property, and the system has the weak periodic orbit closing property.
\end{thmx}

The weak specification property for a flow is a natural analogue of a well known discrete-time definition, and is a weakening of Bowen's original specification property. 
We obtain this property, which is the main point of the theorem above, using geometric arguments. We exploit the existence of a coding of the geodesic flow
due to Gromov \cite{gromov}, and expanded upon by Coornaert and Papadopoulos \cite{cp}, which uses topological arguments to give a suspension on a subshift of finite type $\Susp(\Sigma, \sigma)$, 
and an orbit semi-equivalence $h:\Susp(\Sigma, \sigma)\rightarrow GX$. This gives a ``weak'' symbolic description of $GX$: unlike the semi-conjugacy with a suspension flow which occurs in the negatively curved Riemannian setting, a priori, orbit semi-equivalence is too weak a relationship to preserve any of the refined dynamical properties studied in this paper \cite{GM, KT17}. Our approach is to combine this weak symbolic description with a geometric argument to ``push down" the weak specification property from $\Susp(\Sigma, \sigma)$ to $GX$. The weak periodic orbit closing property, defined in \S \ref{closing}, is obtained using the same philosophy. The expansivity property of the flow is obtained by a simple geometric argument. In general, specification and expansivity are not sufficient to ensure that H\"older continuous potentials have Bowen's regularity property. 
However, we can guarantee this in the $\CAT(-1)$ setting using geometric properties of geodesics in negatively curved spaces.  

Our argument for the weak specification property also applies to some $\CAT(0)$ examples, including all those whose geodesic flow is orbit equivalent to geodesic flow on a $\CAT(-1)$ space. Conversely, in many $\CAT(0)$ cases, it is easy to see that weak specification does not hold, and we can use this to rule out the existence of an orbit semi-equivalence with a compact shift of finite type. We collect these partial results for the $\CAT(0)$ case in \S \ref{sec:CAT(0)spaces}.

In the second part of the paper, we use the characterization of the geodesic flow as an expansive flow with weak specification to study thermodynamic formalism and large deviations for $\CAT(-1)$ spaces. We carry this out using purely analytic arguments, and we obtain the following:

\begin{thmx}\label{thm:applications}
Let $X$ be a compact, locally $\CAT(-1)$, geodesic metric space, with fundamental group not isomorphic to $\mathbb Z$, and $\varphi$ a H\"older
continuous function on $GX$. Then
\begin{enumerate}
\item the potential function $\varphi$ has a unique equilibrium measure $\mu_\varphi$,
\item the equilibrium measure $\mu_\varphi$ satisfies the Gibbs property,
\item the $\varphi$-weighted periodic orbits for the geodesic flow equidistribute to $\mu_\varphi$,
\item the ergodic measures are entropy dense in the space of flow-invariant probability measures,
\item the measure $\mu_\varphi$ satisfies the large deviations principle.
\end{enumerate}
In particular, for the special case $\varphi\equiv 0$, we see that the Bowen-Margulis measure $\mu_{BM}$
is the unique measure of maximal entropy, that $\mu_{BM}$ satisfies the Gibbs property, and that
it satisfies the large deviations principle.
\end{thmx}

The dynamical notions that appear in the above theorem (equilibrium measures, entropy density, large deviations principle, etc.) are 
defined in \S \ref{sec:applications}. In Theorem \ref{thm:applicationsgeneral},  we state and prove our results on thermodynamic formalism and large deviations for the class of expansive flows with weak specification and  potential functions $\varphi$ with the Bowen property. In light of Theorem \ref{thm:weak-specification}, the statement of Theorem \ref{thm:applications} thus follows immediately from Theorem \ref{thm:applicationsgeneral}. 
Technical care must be taken when extending results on flows with specification to the case of weak specification. 
We take particular care in our proof of entropy density of ergodic measures, which is a key step for our large deviations result. To the best of our knowledge, this property has not been studied in the continuous-time setting before, and a self-contained and detailed proof is required. There has been a recent increase in interest in the density and entropy density of ergodic measures \cite{GK, CS2, GP}. In particular, Gorodetski and Pesin \cite{GP} have studied entropy density for $C^{1+\alpha}$ diffeomorphisms using a version of the Katok horseshoe theorem for non-ergodic hyperbolic measures. However, this approach fundamentally belongs to the smooth theory, so even a continuous-time version of this result would not be applicable in the $\CAT(-1)$ setting.

For the geodesic flow on Riemannian manifolds of negative curvature, and more generally for Axiom A flows, uniqueness of equilibrium states for H\"older potentials was proved by Bowen and Ruelle \cite{Bowen-Ruelle}. For expansive flows with strong specification, this result was obtained by Franco \cite{eF77} for potentials with the Bowen property. For geodesic flow on locally $\CAT(-1)$ spaces, the Bowen-Margulis measure, which is defined using the Patterson-Sullivan construction of a measure on the sphere at infinity, has been studied extensively \cite{roblin, LL10}. This measure is well known to be a measure of maximal entropy (MME), as shown by Kaimanovich in the Riemannian setting \cite{vK90, vK91}, and equidistribution of periodic orbits to the Bowen-Margulis measure was shown by Roblin \cite[Theorem 5.1.1]{roblin}. However, uniqueness of the  Bowen-Margulis measure as an MME has not been addressed explicitly until this work, and the large deviations principle for this measure is also new. 

The argument for obtaining the large deviations principle from the specification property goes back to the 1990s with notable references including \cite{De92, Yo, EKW, waddington}.  We adapt this approach to the current setting. Large deviations in dynamical systems were  first developed by Orey and Pelikan \cite{OP} in analogy to results in Probability Theory, see \cite{Ellis}. Large deviations results for flows and semi-flows with weak specification have also been announced in the preprint \cite{BV}.

Uniqueness of equilibrium states beyond the negative curvature compact Riemannian case has received continued interest. For non-positively curved Riemannian manifolds, uniqueness of the MME was proved in the deep work of Knieper \cite{knieper1, knieper2}. Results on the growth rate of weighted regular periodic orbits were obtained in \cite{ GS14}. Recent progress on equilibrium states and weighted equidistribution of periodic orbits in this setting has been made by Burns, Climenhaga, Fisher and the third named author \cite{BCFT}. 

A beautiful theory of equilibrium states has been developed in the non-compact negative curvature  Riemannian setting by Paulin, Pollicott and Schapira \cite{PPS}, including results on uniqueness and equidistribution. In \cite{PPS}, they explicitly state that the reason they assume a smooth structure is due to the difficulties associated with controlling a H\"older potential function on $GX$ for a $\CAT(-1)$ space. We sidestep these difficulties, providing techniques to handle H\"older potentials in the $\CAT(-1)$ setting. This is an advantage of our approach. The results on uniqueness of equilibrium states and weighted equidistribution of periodic orbits are new in the $\CAT(-1)$ setting beyond the Riemannian case. 

 We note that progress towards building a theory of Gibbs measures in the $\CAT(-1)$ setting has also been made recently by Broise-Alamichel, Parkkonen and Paulin in a book project \cite{BPP16} that appeared on the arXiv after the first version of our paper was completed. Their approach has the advantage that it also handles the non-compact case, yielding that Gibbs measures for a restricted class of H\"older potentials are unique when they exist. Their approach requires that the potential is well-defined and well-behaved on an analogue of the unit tangent bundle (see \S2.4 and \S3.2 of \cite{BPP16}). This assumption means that if two geodesics agree for a short time before diverging, the potential (considered on $GX$) must have the same value on each of them. In the non-Riemannian case, this heavily restricts the class of potentials under consideration. When the space is a metric graph of finite groups, their results apply to all H\"older continuous potentials which are well-defined on the unit tangent bundle, and they add to the thermodynamic picture by using countable state symbolic dynamics to show that the unique Gibbs measure is the unique equilibrium state. 
Our method is completely different, and allows us to include the geodesic flow for a compact $\CAT(-1)$ space in the general framework of expansive flows with weak specification. This gives a systematic viewpoint to study the thermodynamic formalism of these flows, and has the major advantage that we can consider H\"older potentials on the space of geodesics without further restrictions. Thus, in the compact setting, we obtain our results for a larger class of potentials, and we prove some results such as entropy density of ergodic measures and the large deviations principle, which are not explored in \cite{BPP16}.

The paper is organized as follows. In \S\ref{sec:spec}, we summarize background material. In \S \ref{sec:theorem1}, we give our geometric argument for the weak specification property. In \S \ref{sec:expansivity}, we prove the other properties of geodesic flows  stated in Theorem \ref{thm:weak-specification}.  In \S \ref{sec:applications}, we prove Theorem \ref{thm:applications} by establishing thermodynamic formalism for expansive flows with weak specification. 

\vskip 10pt

\subsection*{Acknowledgments} We would like to thank the anonymous referees and Tianyu Wang for their helpful comments which have greatly benefited this article. 

\vskip 10pt

%

\section{Background Material} \label{sec:spec}


\subsection{Specification for flows} \label{sec:specf}

Let $\mathcal F=\{f_s\}_{s\in \mathbb R}$ be a continuous flow on a compact metric space $(X,d)$. Given any $t>0$, we can define a new metric by
\[
d_t(x,y) = \max\{d(f_sx, f_sy) : s \in [0, t]\}.
\]

We view $X\times [0,\infty)$ as the space of finite orbit segments for $(X,\mathcal F)$ by associating to each pair $(x,t)$ the orbit segment $\{f_s(x) \mid 0\leq s< t\}$.  

We say that $\mathcal F$ has \emph{weak specification at scale $\delta$} if there exists $\tau>0$ such that for every collection of finite orbit segments  $\{(x_i,t_i)\}_{i=1}^k$, there exists a point $y$ and a sequence of \emph{transition times} $\tau_1,\dots,\tau_{k-1} \in [0,\tau]$ such that for $s_j = \sum_{i=1}^{j} t_i + \sum_{i=1}^{j-1}\tau_i$ and $s_0 = \tau_0 = 0$, we have
\begin{equation}\label{eqn:spec}
d_{t_j}(f_{s_{j-1}+\tau_{j-1}}y, x_j) < \delta \text{ for every } 1\leq j\leq k.
\end{equation}
We say $\mathcal F$ has \emph{weak specification} if it has weak specification at every scale $\delta>0$. We say $\mathcal F$ has \emph{weak specification at scale $\delta$ with maximum transition time $\tau$} if we want to declare a value of $\tau$ that plays the role described above. 
This definition of weak specification for flows appeared recently in the literature in \cite{ClimenhagaThompsonflows}, and under the name `gluing orbit property' in \cite{BV}. 

Intuitively, \eqref{eqn:spec} means that there is some point $y$ whose orbit shadows the orbit of $x_1$ for time $t_1$, then after a transition period which takes time at most $\tau$, shadows the orbit of $x_2$ for time $t_2$, and so on.  Note that $s_j$ is the time spent for the orbit $y$ to approximate the orbit segments $(x_1, t_1)$ up to $(x_j, t_j)$.  It is sometimes convenient to use the word `shadowing' formally: For $y\in X$ and $s\in \mathbb R$, we say that $f_sy$ \emph{$\delta$-shadows} the orbit segment $(x,t)$ if $d_t(f_sy, x)< \delta$. 

The weak specification property clearly implies topological transitivity. Transitivity alone allows us to find an orbit which shadows a finite collection of orbit segments, but it does not give us any control on the length of the transition time. This is the crucial additional ingredient provided by weak specification: the transition times are uniformly bounded above, depending only on the scale $\delta$, and not on the orbit segments, or their lengths.

The specification property for flows which was originally introduced by Bowen is substantially stronger than weak specification. The approximating orbit $y$ is required to be periodic, and the transition times $\tau_i$ are required to be close to $\tau$. See \cite[\S 18.3]{KH} or \cite{bowen-periodic} for the precise definition of this property. 

Finally, we note that while the weak specification property only involves approximating {\it finitely} many orbit segments, it is not difficult to show that this implies the ability to approximate {\it infinitely} many orbit segments. Since we will require this in the proof of Theorem \ref{thm:applications}, details are given in \S \ref{sec:hdense}.


\subsection{Specification for discrete-time systems} \label{def:specmaps} 

Now let $f$ be a continuous map on a compact metric space $X$. We view $X\times \mathbb N$ as the space of finite orbit segments for $(X,f)$ by associating to each pair $(x,n)$ the orbit segment $\{f^ix \mid i \in\{0, \ldots n-1\}\}$. We say that $f$ has \emph{weak specification at scale $\delta$} if there exists $\tau \in \mathbb N$ such that for every collection of finite orbit segments  $\{(x_i,n_i)\}_{i=1}^k$, there exists a point $y$ and a sequence of \emph{transition times} $\tau_1,\dots, \tau_{k-1} \in \mathbb N$ with $\tau_i \leq \tau$ such 
that for $s_j = \sum_{i=1}^{j} n_i + \sum_{i=1}^{j-1}\tau_i$ and $s_0 = \tau_0 = 0$, we have
\begin{equation}\label{eqn:spec2}
d_{t_j}(f^{s_{j-1}+\tau_{j-1}}y, x_j) < \delta \text{ for every } 1\leq j\leq k.
\end{equation}
We say $f$ has \emph{weak specification} if it has weak specification at every scale $\delta>0$. We say $f$ has \emph{specification} if in addition all transition times $\tau_i$ can be taken to be exactly $\tau$ (which depends on $\delta$). 
Classic reference texts for the specification property in discrete-time include \cite{DGS, KH, N}.


\subsection{Shift spaces} We recall some basic properties of shift spaces, referring the reader to \cite{LM, PP} for more details.
The full two-sided shift $\Sigma_\AAA$ on a finite alphabet $\mathcal{A}$ is the space of bi-infinite 
sequences $\mathcal A^{\mathbb Z}$ equipped with the shift operator $\sigma: \Sigma_\AAA \to \Sigma_\AAA$ defined by $\sigma(x)_n = x_{n+1}$ for $(x_n)_{n=-\infty}^\infty \in \Sigma_\AAA$. The space $\Sigma_\AAA$ is 
endowed with the usual product topology, is compact, and is equipped with the metric 
\[d(x, y) = \left\{ \begin{array}{ll} \frac{1}{2^i} \mbox{ where } i=\min\{|n|: x_n \neq y_n\} & \mbox{ when } x\neq y \\
						0 & \mbox{ when } x=y. \end{array} \right.\]

A \emph{shift space}  $(\Sigma, \sigma)$ is a closed, shift-invariant subset $\Sigma$ of $\Sigma_\AAA$ equipped with the shift operator. A shift of finite type (SFT) is a shift space which can be described by a finite set of forbidden words, i.e. words which do not appear in the shift space. 
Given a shift space $(\Sigma, \sigma)$, the \emph{language of 
$\Sigma$}, denoted by $\mathcal L = \mathcal L(\Sigma)$, is the set of all finite words that appear in elements of $\Sigma$. Given $w\in\mathcal L$, let $|w|$ denote the length of $w$. 
The weak specification property has a simpler characterization for shift spaces. It is a straightforward exercise to show that $(\Sigma, \sigma)$ has \emph{weak specification}  in the sense of \S \ref{def:specmaps} if and only if there exists $\tau \in \mathbb N$ so for every $v,w\in \mathcal L(\Sigma)$ there is $u\in \mathcal L(\Sigma)$ such that $vuw\in \mathcal L(\Sigma)$ and $|u|\leq \tau$.


\subsection{Suspension flow} We recall the definition of the suspension flow.

\begin{defn}\label{suspension-flows}
Let $(X,f)$ be a discrete-time dynamical system. Then $\Susp(X,f)$ is the space $(X\times [0,1])/\sim$ 
where $(x,1)\sim(fx,0)$, equipped with the flow $\{\phi_t\}$ defined locally by $\phi_t(x,s) = (x, s+t)$.
\end{defn}
We equip the space with the Bowen-Walters metric \cite{BW}. For two point $(x,s), (y, s)$, we define the horizontal distance to be 
\[
d_H ((x,s), (y,s)) = (1-s) d(x,y) + s d(fx, fy).
\]
For two points $(x, s), (x, t)$, we define the vertical distance to be
\[
d_V((x, s), (x, t)) = |s-t|.
\]
We define $d((x,s), (y,t))$ to be the smallest path length of a chain of horizontal and vertical paths  connecting $(x,s)$ and $(y, t)$, where path length is calculated using $d_H$ and $d_V$. The reason that we use this metric over a more naive choice is that the suspension flow is continuous in the Bowen-Walters metric. We now show that transitivity and weak specification are equivalent for a suspension of an SFT.

\begin{prop}\label{prop:shift spec}
Let $\Sigma$ be a subshift of finite type. The following are equivalent.
\begin{enumerate}
\item $\Sigma$ is transitive;
\item $\Sigma$ satisfies the weak specification property;
\item $\Susp(\Sigma, \sigma)$ is transitive;
\item $\Susp(\Sigma, \sigma)$ satisfies the weak specification property.
\end{enumerate}

\end{prop}
\begin{proof}
We prove (1)$\implies$(2)$\implies$(4)$\implies$(3)$\implies$(1).  

Proving (1)$\implies$(2)  is a straightforward exercise: transitivity for a shift of finite type allows us to transition from any symbol $i$ to another symbol $j$ in bounded time. Thus, to glue two words $v, w \in \mathcal L$, it suffices to look at the final symbol of $v$ and the first symbol of $w$ and take a word which transitions between them.

To prove (2)$\implies$(4), we show that if $(X, f)$ is a dynamical system with the weak specification property, then $\Susp(X, f)$ satisfies weak specification. Suppose $(X, f)$ has weak specification at scale $\delta$ with maximum transition time $\tau$. Suppose that we wish to find an orbit for the suspension flow which approximates the orbit segments $((x_1, s_1), t_1)$, \ldots, $((x_k, s_k), t_k)$ at scale $\delta$. We can apply the weak specification property to approximate the orbit segments $(x_1, \lfloor t_1\rfloor+2)$, \ldots, $(x_k, \lfloor t_k\rfloor+2)$ in the base with an orbit segment $(y, n)$.  It is straightforward to check that if $y\in B_n(x, \delta)$ in the base, then $(y,s) \in B_{n-1}((x, s), \delta)$ in the Bowen-Walters metric.  Using this fact, we can verify that the orbit segment for the flow starting at $(y,s_1)$ approximates the orbit segments $((x_1, s_1), t_1)$, \ldots, $((x_k, s_k), t_k)$ in the sense of \eqref{eqn:spec2} as required, with maximum transition time $\tau+2$.

(4)$\implies$(3) is trivial. All that remains is to show that (3)$\implies$(1), and we prove the contrapositive. If $\Sigma$ is not transitive, then there exists cylinder sets $[w_1], [w_2]$ so that $\sigma^k[w_1] \cap [w_2] = \emptyset$ for all $k$. Clearly, the open sets $A = [w_1] \times (0, \frac12)$, $B = [w_2] \times (0, \frac12)$ satisfy $\phi_tA \cap B = \emptyset$ for all $t$, so $\Susp(\Sigma, \sigma)$ is not transitive.
\end{proof}


\subsection{Orbit equivalence of flows.}  

Let $(X,\{f_s\})$ and $(Y,\{g_s\})$ be continuous flows on compact metric spaces. We recall:

\begin{defn}
A flow $(Y,\{g_s\})$ is \emph{orbit semi-equivalent} to a flow $(X,\{f_s\})$ if there is a continuous surjection $h:X\to Y$, whose restriction to any $\{f_s\}$-orbit in $X$ is an orientation-preserving local homeomorphism onto a $\{ g_s \}$-orbit in $Y$. The flows are \emph{orbit equivalent} if $h:X\to Y$ is a homeomorphism.
\end{defn}
Orbit semi-equivalence is too weak a relationship to preserve any refined dynamical information.   
In particular, weak specification is not preserved by orbit equivalence in general.  To see this, a convenient source of examples of orbit equivalences comes from considering suspension flows with a non-constant roof function $r: X\rightarrow \mathbb (0,\infty)$ over a discrete dynamical system $(X, f)$. It is clear that any two suspension flows over the same base space are orbit equivalent. It is possible to construct a suspension flow over the full shift with more than one measure of maximal entropy, which rules out the possibility that this flow has weak specification. This construction is given in \cite{KT17}.

 Let $h:X\to Y$ be a continuous orbit semi-equivalence between $(X,\{f_s\})$ and $(Y,\{g_s\})$. We prove a result on how orbit semi-equivalence acts on orbit segments which we will use in our proof of the specification property. By continuity of the orbit semi-equivalence, an orbit segment $(x,t)$ for $(X,\{f_s\})$ is mapped to an orbit segment $(h(x), \tau(x,t))$ for $(Y,\{g_s\})$. That is,
\[
h (\{f_s(x) : s \in [0, t]\}) = \{g_s (h(x)) : s \in[0, \tau(x,t)]\},
\]
and in particular, $h(f_t(x))=g_{\tau(x,t)}(h(x))$.

\begin{prop}\label{prop:cts}
Let $(X,\{f_s\})$ and $(Y,\{g_s\})$ be continuous flows on compact metric spaces, and suppose that $(Y,\{g_s\})$ has no fixed points. Let $h:X\to Y$ be a continuous orbit semi-equivalence. Then the function $\tau: X\times [0,\infty) \to [0,\infty)$ defined as above is continuous.
\end{prop}
\begin{proof}
It is clear from continuity of the orbit semi-equivalence that as $s \to t$, $\tau(x, s) \to \tau(x, t)$, so it suffices to study the first coordinate and show that for a fixed $t$, if $x_k \to x$, then $\tau(x_k, t) \to \tau(x, t)$.

We fix $\epsilon>0$. Since the flow $(Y,\{g_t)\}$ has no fixed points, there exists $\delta>0$ so that if $d(g_{s_1}y, g_{s_2}y)<\delta$, then $|s_1-s_2|<\epsilon$.  
Let $\tau:= \tau(x,t)$. Then, by continuity of the flow and $h$, we have $g_\tau(h(x_k)) \to g_\tau(h(x))$. Thus, for $k$ large, we have
\[
d(g_\tau(h(x_k)), g_\tau(h(x)))< \delta/2,
\]
where $d$ is the metric on $Y$. Now we consider the sequence $h(f_tx_k)$. By continuity, $h(f_tx_k) \to h(f_tx)=g_\tau(h(x))$. Thus, for $k$ large, we have
\[
d(h(f_tx_k), g_\tau(h(x))) < \delta/2,
\]
 and so we have $d(g_{\tau(x_k, t)}(h(x_k)), g_\tau(h(x_k))) =   d(h(f_tx_k), g_\tau(h(x_k)))<\delta$, and these points are on the same orbit. Thus it follows that $|\tau(x_k, t)-\tau|<\epsilon$. It follows that $\tau(x_k, t) \to (x, t)$, and thus the function $\tau$ is continuous.
\end{proof}
\begin{cor} \label{boundedtransition}
Let $(X,\{f_t\})$, $(Y,\{g_t\})$, and $h:X\to Y$ be as in Proposition \ref{prop:cts}. Then for all $t$, there exists $\kappa=\kappa(t)>0$, so that for all $x \in X$, the image of $(x,t)$ under $h$ is contained in the orbit segment $(h(x), \kappa)$. That is,
\[
h (\{f_s(x) : s \in [0, t]\}) \subset \{g_s (h(x)) : s \in[0, \kappa]\}.
\]
\end{cor}
\begin{proof}
By continuity of $\tau$, and compactness of $X \times\{t\}$, $\sup\{\tau(x,t): x \in X\}<\infty$.
\end{proof}


\subsection{$\CAT(-1)$ spaces and their geodesic flows.} 

We now recall some basic results on the geometry and dynamics of locally $\CAT(-1)$ space. A detailed discussion of the geodesic flow on locally $\CAT(-1)$ spaces can be found in Ballmann's book \cite{ballmann} or in Roblin's monograph \cite{roblin}. 
 Given any geodesic triangle $\Delta(x,y,z)$ inside a geodesic space $X$, one can construct a comparison triangle 
$\Delta(\bar x, \bar y, \bar z)$ inside the hyperbolic plane $\mathbb H^2$ having exactly the
same side lengths. Corresponding to any pair of points $p, q$ on the triangle $\Delta(x,y,z)$,
there is a corresponding pair of comparison points $\bar p, \bar q$ on $\Delta(\bar x, \bar y, \bar z)$.
The triangle is said to satisfy the $\CAT(-1)$ inequality if, for every such pair of points, one
has the inequality $d_X(p,q) \leq d_{\mathbb H^2}(\bar p, \bar q)$. A geodesic space is $\CAT(-1)$
if every geodesic triangle in the space is $\CAT(-1)$. It is locally $\CAT(-1)$ if every point has
a neighborhood which is $\CAT(-1)$. Any compact locally $\CAT(-1)$ space $X$ has a universal cover $\tilde X$ which is $\CAT(-1)$, with $\Gamma:=\pi_1(X)$
acting isometrically on $\tilde X$.

The definition for a $\CAT(0)$ space is obtained by replacing $\mathbb{H}^2$ with $\mathbb{R}^2$, the model space of curvature 0, in the above.

To a $\CAT(-1)$ space $\tilde X$, one can associate a {\it boundary at infinity} $\partial ^\infty \tilde X$,
consisting of equivalence classes of geodesic rays $\eta: [0, \infty) \rightarrow \tilde X$, where rays 
are considered equivalent if they remain at bounded distance apart. Note that any geodesic 
$\gamma: \mathbb R\rightarrow \tilde X$ naturally gives rise to a pair of points 
$\gamma ^\pm \in \partial ^\infty \tilde X$. If we form $G\tilde X$ the space of all geodesics in $\tilde X$,
there is thus a natural identification $G\tilde X \cong \big((\partial ^\infty \tilde X \times \partial ^\infty \tilde X) \setminus \Delta\big) \times \mathbb R$, where $\Delta \subset \partial ^\infty \tilde X \times \partial ^\infty \tilde X$ is the diagonal.
There is a natural flow on $G \tilde X$, given by translating in the $\mathbb R$-factor,which we call the \emph{geodesic flow} on $\tilde X$. 
This geodesic flow on $G \tilde X$ can be written as $g_t(\gamma(s))=\gamma(s+t)$.

Now if $X$ is locally $\CAT(-1)$, then one can similarly form the space $GX$ of geodesics in $X$, where
a geodesic is a locally isometric map $\gamma: \mathbb R\rightarrow X$. This comes equipped with a 
natural flow, given by pre-composing by translations on $\mathbb R$, which we call the {\it geodesic
flow on $X$}. The fundamental group $\Gamma$ 
acts isometrically on the universal cover $\tilde X$, hence on the boundary at infinity $\tilde X$, and
on the space of geodesics $G\tilde X$. The flow on $G\tilde X$ commutes with the $\Gamma$-action, 
hence descends to a flow on $(G\tilde X)/\Gamma$, and there is a flow equivariant homeomorphism
$GX \cong (G\tilde X)/\Gamma$. 

Finally, if the locally $\CAT(-1)$ space $X$ is compact, then the fundamental 
group $\Gamma$ is a Gromov hyperbolic group, see \cite{gromov}. Such a group has a well-defined boundary
at infinity $\partial ^\infty \Gamma$, and there is a $\Gamma$-equivariant homeomorphism 
$\partial ^\infty \Gamma \cong \partial ^\infty \tilde X$. This allows us to apply results on $\partial ^\infty \Gamma$ 
obtained from the theory of Gromov hyperbolic groups to the boundary $\partial ^\infty \tilde X$.

The space $GX$ of all geodesics in $X$ can be endowed with the following metric:
\[ d_{GX}(\gamma_1, \gamma_2) = \inf_{\tilde \gamma_1, \tilde \gamma_2} \int_{-\infty}^\infty d_{\tilde X}(\tilde \gamma_1(t), \tilde \gamma_2(t))e^{-2|t|} dt \]
where the infimum is taken over all lifts $\tilde \gamma_i$ of $\gamma_i$ to $G\tilde X$.  Since the lifts of a given geodesic form a discrete set on $G\tilde X$, the infimum is in fact a minimum. The factor $2$ in the exponent normalizes the metric so that, for small $s$, $d_{GX}(\gamma, g_s\gamma)=s$.

We assume from now on that the fundamental group $\Gamma=\pi_1(X)$ is {\it non-elementary}, i.e. not isomorphic to $\mathbb Z$. 
This is the generic case. When $\Gamma \cong \mathbb Z$ (e.g. $X=S^1$), the geodesic flow on $X$ behaves differently from other examples, 
and is simple to investigate. $GX$ consists of two disjoint circles, with the flow acting by rotations on the circles.
Note that specification clearly \emph{fails} in this case, as two orbit segments on the distinct circles can never be 
approximated by a single orbit segment. 

We collect some results on $\CAT(-1)$ spaces that we use in this paper. 

\begin{lem} \label{top-transitive}
Let $X$ be a compact, locally $\CAT(-1)$, geodesic metric space. Then the geodesic flow on 
$GX= G(\tilde X/\Gamma) = (G\tilde X)/\Gamma$ is topologically transitive.
\end{lem}

\begin{proof}
Since $\Gamma$ is non-elementary, the $\Gamma$-action on $\partial ^\infty \Gamma$ has dense orbits (see \cite[Section 8.2]{gromov}), and hence
so does the $\Gamma$-action on $\partial ^\infty \tilde X$. The lemma is now an immediate consequence of \cite[Theorem III.2.3]{ballmann}.
\end{proof}

The following result is a key ingredient for our approach, and gives the existence of symbolic dynamics for geodesic flow on $\CAT(-1)$ spaces using a topological construction reminiscent of the Bowen-Series approach. The main point of the proof was sketched by Gromov, and developed in detail by Coornaert and Papadopoulos \cite{cp} for the geodesic flow on a word hyperbolic group.

\begin{prop}\label{symbolic-coding}
Let $X$ be a compact, locally $\CAT(-1)$, geodesic metric space. Then there exists a topologically transitive subshift
of finite type $(\Sigma, \sigma)$, and an orbit semi-equivalence $h:\Susp(\Sigma,\sigma) \to GX$. Moreover,
$h$ is finite-to-one.
\end{prop}

\begin{proof}
To a Gromov hyperbolic group $\Gamma$, one can associate a metric space 
$\hat G (\Gamma)$, equipped with both a $\Gamma$-action, and a $\Gamma$-equivariant 
$\mathbb R$-flow. The space $\hat G(\Gamma)$ is constructed to satisfy certain universal properties. 
The construction was outlined by Gromov in \cite[Theorem 8.3.C]{gromov}, with detailed arguments
worked out by Champetier \cite[Section 4]{champetier} (see also Math\'eus \cite{matheus}). 

The quotient metric space $\bar G(\Gamma):= \hat G(\Gamma) /\Gamma$, equipped with the
induced $\mathbb R$-flow, has an orbit semi-equivalence $h_1: \Susp(\Sigma, \sigma) \rightarrow \bar G(\Gamma)$ 
which is uniformly finite-to-one, where $\Sigma$ is a shift of finite type.
This was explained by Gromov in \cite[Section 8.5.Q]{gromov}, and a careful proof can be found in the
paper by Coornaert and Papadopoulos \cite{cp}. Finally, as noted on \cite[pg. 484, Facts 4 and 5]{cp}, in the 
case where $X$ is compact locally $\CAT(-1)$ and $\Gamma = \pi_1(X)$, one has a $\Gamma$-equivariant
orbit equivalence $G\tilde X \rightarrow \hat G(\Gamma)$ (this is deduced from the universal properties
of the flow space $\hat G(\Gamma)$). This descends to an orbit equivalence $h_2:GX \rightarrow \bar G(\Gamma)$.
Defining $h:= h^{-1}_2 \circ h_1 : \Susp(\Sigma, \sigma) \rightarrow GX$ provides the claimed orbit semi-equivalence. 
To see that $\Sigma$ can be taken to be transitive, we can simply observe that since $h$ is an orbit semi-equivalence onto a transitive flow, we still get an orbit semi-equivalence if we restrict to a suitable transitive component of $\Sigma$.
\end{proof}

The following lemma shows that geodesics which are close in $GX$ are close when evaluated at time $0$ on $X$.
\begin{lem}\label{lem:tool2}
For all $\gamma_1, \gamma_2\in GX$, 
\[ d_X(\gamma_1(0), \gamma_2(0)) \leq 2d_{GX}(\gamma_1,\gamma_2).\]
Furthermore, for $s, t \in \mathbb R$, $d_X(\gamma_1(s),\gamma_2(t)) \leq 2 d_{GX}(g_s\gamma_1,g_t \gamma_2).$
\end{lem}

\begin{proof}
Consider lifted geodesics $\tilde \gamma_1, \tilde \gamma_2 \in G \tilde X$ such that
\[
d_{GX}(\gamma_1, \gamma_2) = d_{G \tilde X} ( \tilde \gamma_1, \tilde \gamma_2) = \int_{-\infty}^\infty d_{\tilde X} ( \tilde \gamma_1(t), \tilde \gamma_2(t))e^{-2|t|} dt.
\]
The function $d_{\tilde X}(\gamma_1(t),\gamma_2(t))$ is a convex function of $t$, and thus for $t\geq 0$ or $t \leq 0$, $d_{\tilde X}(\tilde \gamma_1(t), \tilde \gamma_2(t))\geq d_{\tilde X}(\tilde \gamma_1(0),\tilde \gamma_2(0))$. In either case, we have
\[
d_{G \tilde X} ( \tilde \gamma_1, \tilde \gamma_2) \geq d_{\tilde X}(\tilde \gamma_1(0),\tilde \gamma_2(0)) \int_0^\infty   e^{-2t} dt = \frac{1}{2}d_{\tilde X}(\tilde \gamma_1(0),\tilde \gamma_2(0)).
\]
Noting that $d_X(\gamma_1(0),\gamma_2(0))\leq d_{\tilde X}(\tilde \gamma_1(0),\tilde \gamma_2(0))$ gives the first statement. Observing that $g_s\gamma_1(0) = \gamma_1(s)$ and $g_t\gamma_2(0)= \gamma_2(t)$ and applying the first result completes the proof. 
\end{proof}

For $\gamma \in GX$, we use the notation $\gamma([0, T]) := \{\gamma(s) : s \in [0,T] \}$ for a segment of $\gamma$, considered as a path in $X$. 
We want to lift and compare geodesic segments after a possible time change, so it is convenient to make the following definition.
\begin{defn}
We say that $\rho:[0, T_1] \to [0, T_2]$ is a \emph{time-change function} if it is a continuous, increasing and surjective function.
\end{defn}
Let $\epsilon_0:=\frac{1}{2} \inf \left \{ l(\gamma): \gamma \mbox{ is a closed geodesic} \right \}$, and note that the $\CAT(-1)$ condition and compactness ensure $\epsilon_0>0$.
The following lemma, whose proof is omitted and is a straightforward exercise, shows that geodesic segments that are close (after time change) on $X$ are close after lifting to the universal cover.

\begin{lem}\label{lem:lift track}
Let $\epsilon<\epsilon_0$ and let $\gamma_1([0, T_1])$, $\gamma_2([0, T_2])$ be geodesic segments and $\rho:[0,T_2]\to[0,T_1]$ a time change such that $d_X(\gamma_1(\rho(t)),\gamma_2(t))<\epsilon$ for all $t\in[0,T_2]$. Then for any lift $\tilde \gamma_1$ of $\gamma_1$, there exists a lift $\tilde \gamma_2$ of $\gamma_2$ with $\tilde \gamma_i(0)$ lying above $\gamma_i(0)$ such that $d_{\tilde X}(\tilde \gamma_1(\rho(t)),\tilde \gamma_2(t))<\epsilon$ for all $t\in[0,T_2]$.
\end{lem}
Complementing Lemma \ref{lem:tool2}, the following Lemma shows that geodesic segments which stay close in $X$ are close in $GX$. 

\begin{lem}\label{lem:shadow in X}
Let $\epsilon<\epsilon_0$ be given and $a<b$ arbitrary. Then there exists $T=T(\epsilon)>0$ such that if $d_X(\gamma_1(t),\gamma_2(t))<\epsilon/2$ for all $t\in[a-T,b+T]$, then $d_{GX}(g_t\gamma_1,g_t\gamma_2)<\epsilon$ for all $t\in[a,b]$. For small $\epsilon$, we can take $T(\epsilon)= -  \log(\epsilon)$.
\end{lem}

\begin{proof}
Choose $T=T(\epsilon)$ so that $\int_T^\infty (\epsilon/2 + 2(\sigma-T))e^{-2\sigma}d\sigma < \epsilon/4.$ Analysis of this integral shows that for small $\epsilon$,  we could take $T(\epsilon) = \log(\epsilon^{-1})$. Lift $\gamma_i$ to $\tilde \gamma_i$ with $d_{\tilde X}(\tilde\gamma_1(t), \tilde \gamma_2(t))<\epsilon/2$ by Lemma \ref{lem:lift track}. First, we consider the integral $\int_{a-T}^{b+T} d_{\tilde X}(\tilde \gamma_1(\tau),\tilde \gamma_2(\tau)) e^{-2|\tau-t|} d\tau$ and note that we can bound 
$d_{\tilde X}(\tilde \gamma_1(\tau), \tilde \gamma_2(\tau))$, and thus the whole integral independent of $T$, by $\epsilon/2$.

We now consider the integrals 
\[
\int_{-\infty}^{a-T} d_{\tilde X}(\tilde \gamma_1(\tau), \tilde \gamma_2(\tau)) e^{-2|\tau-t|} d\tau \mbox{ and }\int_{b+T}^{\infty} d_{\tilde X}(\tilde \gamma_1(\tau),\tilde \gamma_2(\tau)) e^{-2|\tau-t|} d\tau.
\]
Since $a\leq t \leq b$, over the domain of the first integral $|\tau-t|=-(\tau-t)$, and over the domain of the second interval $|\tau-t|=(\tau-t)$. 

In the first, we may bound $d_{\tilde X}(\tilde \gamma_1(\tau), \tilde \gamma_2(\tau))<\epsilon/2+2(a-T-\tau)$ and in the second, $d_{\tilde X}(\tilde \gamma_1(\tau), \tilde \gamma_2(\tau))<\epsilon/2+2(\tau-b-T)$ using the triangle inequality. It follows that $d_{GX}(g_t\gamma_1, g_t\gamma_2) =  \int_{-\infty}^{\infty} d_{\tilde X}(\tilde \gamma_1(s+t),\tilde \gamma_2(s+t)) e^{-2|s|} ds$ is bounded above by
\[
\int_{-\infty}^{a-T} (\epsilon/2+2(a-T-\tau)) e^{2(\tau-t)} d\tau + \int_{b+T}^{\infty} (\epsilon/2+2(\tau-b-T)) e^{-2(\tau-t)} d\tau  +  \epsilon/2,
\]
making the change of variables $\tau=s+t$. The first integral is largest when $t=a$, the second when $t=b$. Making these substitutions and changing variables by $\sigma=\tau-a$, $\sigma=\tau-b$, respectively,
\begin{align}
	d_{GX}(g_t\gamma_1, g_t\gamma_2) & <  \int_{-\infty}^{-T} (\epsilon/2+2(T-\sigma)) e^{2\sigma} d\sigma + \int_{T}^{\infty} (\epsilon/2+2(\sigma-T)) e^{-2\sigma} d\sigma \nonumber  +  \epsilon/2. \nonumber
\end{align}

\noindent Our choice of $T$ finishes the proof.
\end{proof}

%

\section{Weak specification for the geodesic flow}\label{sec:theorem1}

We consider a compact, locally
$\CAT(-1)$, geodesic space $X$, and we wish to establish the weak specification property for $GX$.
By Lemma \ref{symbolic-coding}, there exists a topologically transitive subshift of finite type $(\Sigma, \sigma)$, and an orbit
semi-equivalence $h: \Susp(\Sigma, \sigma) \rightarrow GX$. 
On $\Susp(\Sigma, \sigma)$, Proposition \ref{prop:shift spec} shows that transitivity immediately bootstraps to 
weak specification. We now show that this property can be transported to $GX$ using the orbit semi-equivalence $h$.
While the weak specification property is not preserved under a general orbit semi-equivalence, the geometry of our setting provides more structure to carry out our argument.

The following lemma allows us to show that geodesic segments which are close after a time change are in fact close without the time change. This is where the assumption that the geodesic flow is on a space of negative curvature is used. The proof requires only that geodesics in the universal cover are globally length minimizing, so a non-positive curvature assumption would be sufficient. 

\begin{prop}\label{prop:nbhd shadow}
Let $X$ be a $\CAT(-1)$ space, and $\gamma_1, \gamma_2 \in GX$ be geodesics.  Suppose there exists a time change $\rho:[0,T_2] \to [0,T_1]$ so that $d_X(\gamma_1(\rho(t)), \gamma_2(t))< \epsilon$ for all $t \in [0, T_2]$. Then $d_X(\gamma_1(t), \gamma_2(t)) < 3\epsilon$ for all $t \in [0,T_1-2\epsilon]$.
\end{prop}

\begin{proof}
First, using Lemma \ref{lem:lift track}, we lift $\gamma_i$ to geodesic segments on the universal cover so that $d_{\tilde X}(\tilde \gamma_1(\rho(t)),\tilde \gamma_2(t))<\epsilon$ for all $t\in[0,T_2]$. If we prove the statement in the universal cover, we have proven it in the original space. In the universal cover, the geodesics are globally length minimizing, and $d_{\tilde X}(\tilde\gamma_i(t_1),\tilde\gamma_i(t_2))=|t_1-t_2|.$

We fix $t\in [0,T_2]$,  and we know that $\tilde\gamma_2(t)$ is within distance $\epsilon$ of $\tilde\gamma_1(\rho(t))$. Then one can reach $\tilde\gamma_2(t)$ from $\tilde\gamma_2(0)$ by the geodesic $\tilde\gamma_2$, or by following the path $\tilde\gamma_2(0) \to \tilde\gamma_1(0) \to \tilde\gamma_1(\rho(t)) \to \tilde\gamma_2(t)$ (see Figure \ref{fig:track}). By the length-minimizing property of $\tilde\gamma_2$,
\[ 
t = d_{\tilde X}(\tilde\gamma_2(0),\tilde\gamma_2(t)) < 2\epsilon + d_{\tilde X}(\tilde\gamma_1(0),\tilde\gamma_1(\rho(t))) = 2\epsilon + \rho(t). 
\]

\begin{center}

\setlength{\unitlength}{.5pt}

\begin{picture}(600,190)(-300,-70)

\qbezier(-250,50)(0,-20)(250,50)	
\qbezier(-250,-50)(0,20)(250,-50)

\put(-210,40){\circle*{5}}
\put(-210,-40){\circle*{5}}

\put(80,19){\circle*{5}}
\put(80,-19){\circle*{5}}

\put(-265,58){$\tilde\gamma_2(0)$}
\put(-265,-74){$\tilde\gamma_1(0)$}

\put(92,39){$\tilde\gamma_2(t)$}
\put(92,-50){$\tilde\gamma_1(\rho(t))$}

\put(-230,-5){$\Bigg\{$}
\put(-260,-5){$<\epsilon$}

\put(85,-5){$\Big\}$}
\put(100,-5){$<\epsilon$}

\qbezier(-210,60)(0,60)(80,60)
\qbezier(-210,53)(-210,60)(-210,67)
\qbezier(80,53)(80,60)(80,67)

\put(-80,70){$t$}

\qbezier(-210,-60)(0,-60)(80,-60)
\qbezier(-210,-53)(-210,-60)(-210,-67)
\qbezier(80,-53)(80,-60)(80,-67)

\put(-90,-80){$\rho(t)$}

\end{picture}

\begin{figure}[h]
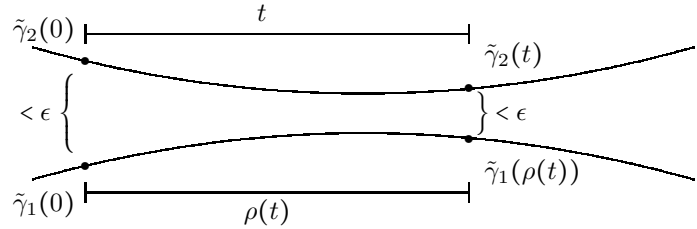

\caption{Nearby geodesics in the $\textrm{CAT}(-1)$ space $\tilde X$ must shadow each other.}\label{fig:track}
\end{figure}

\end{center}

\noindent By interchanging the roles of the geodesics, $\rho(t) < 2\epsilon + t$, and so $|t-\rho(t)|<2\epsilon.$ Thus, 
\begin{align*}
d_{\tilde X}(\tilde \gamma_1(t),\tilde \gamma_2(t))& \leq d_{\tilde X}(\tilde \gamma_1(t),\tilde \gamma_1(\rho(t)))+ d_{\tilde X}(\tilde \gamma_1(\rho(t)),\tilde \gamma_2(t))\\ & \leq |t-\rho(t)| + \epsilon < 3\epsilon.
\end{align*}
Since $d_{\tilde X}(\tilde \gamma_1(T_1), \tilde \gamma_2(T_2))<\epsilon$, a similar argument shows that $|T_1-T_2|<2\epsilon.$ Thus, the above estimate holds for $t \in [0,T_1-2\epsilon]$.
\end{proof}

The proof of the weak specification property for geodesic flow on a compact $\CAT(-1)$ space is an immediate corollary, via Proposition 
\ref{top-transitive} and Lemma \ref{symbolic-coding}, of the following result. 

\begin{thm} \label{thm:spec}
Suppose that $(Y, \FFF)$ is a flow on a compact space satisfying the weak specification property. 
Suppose that $h:Y \to GX$ is a continuous, surjective orbit semi-equivalence to the geodesic flow on a
compact, locally $\CAT(-1)$ space $X$. Then the geodesic flow $(GX, \{g_t\})$ satisfies the weak specification property.
\end{thm}

\begin{proof}
Let $\epsilon>0$. We fix a collection of orbit segments $\{(\gamma_i,t_i)\}_{i=1}^k$ for $(GX, \{g_t\})$, and show how to glue them together.  Let $T=T(\epsilon)$  be the constant from Lemma \ref{lem:shadow in X}. As $h$ is uniformly continuous, let $\delta>0$ be so small that $d_Y(y_1,y_2)<\delta$ implies $d_{GX}(h(y_1), h(y_2))<\epsilon/6$. Thus, writing $\gamma_1=h(y_1), \gamma_2=h(y_2)$, it follows from Lemma \ref{lem:tool2} that $d_X(\gamma_1(0), \gamma_2(0))< \epsilon/3$. 

Fix lifts $\{(y_i,\hat t_i)\}_{i=1}^k$ under $h$ of orbit segments $\{(g_{-T}\gamma_i, t_i+2\epsilon+2T)\}_{i=1}^k$. That is, each $(y_i,\hat t_i)$ is an orbit segment for $(Y, \FFF)$ such that 
\[
\{h(f_s y_i): s \in [0, \hat t_i]\} = \{g_s \gamma_i : s \in[-T, t_i+T+ 2 \epsilon]\}.
\]
The first step is to apply the specification property to these lifted orbit segments. Let $\hat \tau$ be provided by the weak specification property for $(Y, \FFF)$ at scale $\delta$. There is a point $z\in Y$ and a sequence of transition times $\hat \tau_1, \ldots \hat \tau_{k-1} \leq \hat \tau$ such that 
\[ d_{\hat t_j}(f_{\hat s_{j-1}+\hat \tau_{j-1}}z,y_j) < \delta \mbox{ for every } 1\leq j\leq k,\] 
where $\hat s_j=\sum_{i=1}^j \hat t_i + \sum_{i=1}^{j-1}\hat \tau_i$. Fix an index $j$, and write $z' = f_{\hat s_{j-1}+ \hat \tau_{j-1}}z$. Consider the image under $h$ of the orbit segment $(z',\hat t_j)$. Then for all $s \in [0, \hat t_j]$, 
\[
d_{GX}(h(f_s z'), h(f_s y_j)) < \epsilon/6.
\]
Thus, writing $h(z')=\gamma'$ and reparameterizing, we see there is a time change $\rho$ so that for all $s \in [0, t_j + 2 \epsilon+2T]$, 
\[
d_{GX}(g_{\rho(s)} \gamma', g_s (g_{-T} \gamma_j)) < \epsilon/6.
\]
Using Lemma \ref{lem:tool2}, we see that for all $s \in [0, t_j + 2 \epsilon+2T]$, 
\[
d_X(\gamma'(\rho(s)), g_{-T}\gamma_j(s)) <\epsilon/3.
\]
Now we apply Proposition \ref{prop:nbhd shadow} to obtain that for all $s \in [0, t_j+2T]$
\[
d_X(\gamma'(s), g_{-T}\gamma_j(s)) <\epsilon.
\]
Next we apply Lemma \ref{lem:shadow in X} to obtain that for all $s \in [T, t_j+T]$, 
\[ 
d_{GX}(g_s\gamma', g_s(g_{-T} \gamma_j)) <2\epsilon,
\]
and thus for all $s \in [0, t_j]$, $d_{GX}(g_s(g_T\gamma'), g_s(\gamma_j)) <2\epsilon$.

Now consider $\gamma = g_T(h(z))$. Noting that $g_T\gamma'$ is an appropriate iterate of $\gamma$ under $(GX, g_t)$, the argument above shows that for each $j$, an appropriate iterate of $\gamma$ is $2 \epsilon$-shadowing for $(\gamma_j, t_j)$. 

It only remains to show that the transition times for $\gamma$ remain controlled. We appeal to Corollary \ref{boundedtransition}, which shows there exists $\kappa$ so that for all $y\in Y$, the image of an orbit segment $(y, \hat \tau)$ under the orbit equivalence $h$ is contained in the orbit segment $(h(y), \kappa)$. 
The segments of $\gamma$ that correspond to transitions between the shadowed orbit segments comprise of images of orbit segments of the form $(y, \hat \tau_i)$ with $\hat \tau_i \leq \hat \tau$, and an additional run of length at most $2T$ coming from the application of Lemma \ref{lem:shadow in X}. Thus the transition times are bounded above by $\kappa+2T$. It follows that $(GX,\{g_t\})$ satisfies weak specification.
\end{proof}


\subsection{Geodesic flow on $\CAT(0)$ spaces}\label{sec:CAT(0)spaces} We now briefly consider the case of non-positive curvature. 

\begin{thm} \label{thm:CAT0}
Let $X$ be a compact, locally $\CAT(0)$, geodesic metric space with fundamental group not isomorphic to $\mathbb Z$ and topologically transitive geodesic flow. If there exists an orbit semi-equivalence $h: \Susp(\Sigma, \sigma)\rightarrow GX$, where $(\Sigma, \sigma)$ is a compact subshift of finite type, then the geodesic flow on $GX$ satisfies the weak specification property.
\end{thm}

We observe that this follows from the proof given in the previous section, where we used the assumption of $\CAT(-1)$ in only two places; the first was to provide the orbit-equivalent symbolic description of $GX$ (Proposition \ref{symbolic-coding}), which we now \emph{assume} to hold; the second was in the proof of Proposition \ref {prop:nbhd shadow} and we already observed that a $\CAT(0)$ assumption was sufficient for that argument. We conclude that our proof also gives the statement of Theorem \ref{thm:CAT0}.

A class of examples that is covered by Theorem \ref{thm:CAT0} is given by $\CAT(0)$ spaces whose geodesics can be mapped homeomorphically to the geodesics for a $\CAT(-1)$ metric. For example, on a Riemannian surface with genus at least $2$, non-positive curvature metrics can be found so that a single closed geodesic has curvature zero, and geodesics can be mapped homeomorphically to those for a hyperbolic metric.  Such examples are clearly expansive, although we can no longer conclude that H\"older potentials have the Bowen property (see Section \ref{sec:expansivity}).

We can also \emph{rule out} orbit semi-equivalence to a suspension of a shift of finite type in many cases. Let $X$ be a compact, locally $\CAT(0)$ metric space. We say that $\tilde X$ has a \emph{fat $1$-flat} if there exists a geodesic $\gamma$ such that for some $w>0$ the $w$-neighborhood $U = N_w(\gamma)$ of $\gamma$ splits isometrically as $\mathbb R \times Y$. An example of such a space is a Riemannian manifold with non-positive sectional curvature which has an open neighborhood $U$ of a closed geodesic where the sectional curvature is identically zero. See \cite{CLMT} for a study of Riemannian manifolds that admit fat flats, and \cite{CS} for many negative results on hyperbolic-type properties in the special case of Riemannian surfaces which have an embedded flat cylinder. 
We show:
\begin{thm}\label{thm:obstruction}
Let $X$ be a compact locally $\CAT(0)$ metric space with topologically transitive geodesic flow such that $\tilde X$ admits a fat $1$-flat. Then 
\begin{enumerate}
\item the geodesic flow $(GX,\{g_t\})$ does not satisfy weak specification; 
\item there does not exist an orbit semi-equivalence $h: \Susp(\Sigma, \sigma)\rightarrow GX$, where $(\Sigma, \sigma)$ is a compact subshift of finite type.
\end{enumerate}
\end{thm}

\begin{proof}
Suppose that $(GX,\{g_t\})$ satisfies weak specification. Let $\delta=\frac{w}{20}$, and let $\tau(\delta)$ be the corresponding maximum transition time. Take a geodesic $\gamma$ and $w>0$ be such that $N_w(\gamma)$ splits isometrically as $\mathbb R \times Y$. Let $\gamma_1=\gamma$ and $\gamma_2$ be a geodesic with $\gamma_2(0)\notin N_w(\gamma)$.  Let $t_1 = \tau$ and $t_2=1$. For the weak specification property to hold in $GX$, there must be some geodesic $\gamma^*$ which $\delta$-shadows $\gamma$ for time $t_1$, then after transition time at most $\tau$, $\delta$-shadows $\gamma_2$.

By Lemma \ref{lem:tool2}, $d(\gamma(t),\gamma^*(t))<2\delta=w/10$ for all $t\in [0,t_1]$. By the geometry of the flat neighborhood $N_w(\gamma)$, $\gamma^*(t)$ travels at most distance $w/5$ perpendicular to the image of $\gamma$ over $t\in [0,t_1]$, remaining all the while in the $w/10$-neighborhood of $\gamma$. Therefore, over the subsequent $\tau=t_1$ units of time, it can again travel at most distance $w/5$ perpendicularly away from the image of $\gamma$. Therefore at any time $t\in [\tau, 2\tau]$, $\gamma^*(t)$ is at least distance $w/5$ from $\gamma_2(0)$. To fulfill the desired shadowing, for some such $t$, $g_t\gamma^*$ should be within $\delta$ of $\gamma_2$. At such a time, $d_{GX}(g_t\gamma^*, \gamma_2)<\delta=\frac{w}{20}$. Using Lemma \ref{lem:tool2}, we must at this point have $d(\gamma^*(t),\gamma_2(0)) <2\delta = \frac{w}{10}$. This is a contradiction, so $\gamma^*$ cannot achieve the shadowing required. We have shown that $(GX,g_t)$ cannot have the weak specification property. 

Now suppose there is an orbit semi-equivalence $h:\Susp(\Sigma, \sigma) \to GX$, where $(\Sigma, \sigma)$ is a shift of finite type. Restricting $\Sigma$ to a transitive component $\Sigma'$ such that $h: \Sigma' \to GX$ is surjective, the arguments of \S \ref{sec:theorem1} show that $(GX,\{g_t\})$ has weak specification. This is a contradiction, so no such $h:\Susp(\Sigma, \sigma) \to GX$ exists.
\end{proof}

Theorem \ref{thm:obstruction} rigorously confirms the expected phenomenon that a compact shift of finite type can not capture the dynamics of this setting. Beyond uniform hyperbolicity, the best hope to capture the dynamics symbolically is often to code the region of the space that experiences `some' hyperbolicity using a shift of finite type on a countable alphabet. The existence of this kind of symbolic dynamics for smooth flows on three dimensional Riemannian manifolds was established by Lima and Sarig \cite{LS}. This kind of phenomenon is not ruled out by Theorem \ref{thm:obstruction}.

%

\section{Expansivity, the Bowen property, and orbit closing}\label{sec:expansivity}

Before turning to applications of the weak specification property, we require three further properties of the geodesic flow on  a compact $\CAT(-1)$ space.


\subsection{Expansivity} The first property we want to check is expansivity. We say a continuous flow $(X, \FFF)$ is \emph{expansive} if for all $\epsilon>0$, there exists $\delta>0$ such that for all $x,y\in X$ and all continuous $\tau:\mathbb{R}\to \mathbb{R}$ with $\tau(0)=0$, if $d(f_t(x), f_{\tau(t)}(y))<\delta$ for all $t\in \mathbb{R}$, then $y=f_s(x)$ for some $s$, where $|s|< \epsilon$.

\begin{prop}\label{prop:expansive}
The geodesic flow on a compact $\CAT(-1)$ space is expansive.
\end{prop}

\begin{proof}
Consider any $\tau:\mathbb{R}\to\mathbb{R}$ with $\tau(0)=0$. Suppose that $\gamma_1, \gamma_2 \in GX$ with $d_{GX}(g_t\gamma_1, g_{\tau(t)}\gamma_2)<\delta$ for all $t$. Then, by Lemma \ref{lem:tool2}, $d_X(\gamma_1(t),\gamma_2(\tau(t)))<2\delta$ for all $t$. By Proposition \ref{prop:nbhd shadow}, it follows that $d_X(\gamma_1(t),\gamma_2(t)) <6\delta$. Choosing $\delta$ so small that $6\delta<\epsilon_0$, we may use Lemma \ref{lem:lift track} and lift the geodesics $\gamma_1$ and $\gamma_2$ to the universal cover in such a way that $d_{\tilde X}(\tilde\gamma_1(t), \tilde\gamma_2(t))<6\delta$ for all $t$. From the definition of the boundary at $\infty$, it follows that $\tilde\gamma_1(\infty) = \tilde\gamma_2(\infty)$ and $\tilde\gamma_1(-\infty) = \tilde\gamma_2(-\infty)$. Hence $\gamma_2(t)=\gamma_1(t+s)$ for some $s$. Since $d_{GX}(\gamma_1, \gamma_2)< \delta$, a straightforward calculation with the definition of $d_{GX}$ implies that given a fixed $\epsilon$, we can choose $\delta$ small enough so that $|s|<\epsilon$.
\end{proof}


\subsection{Bowen property} 

The second property we want is a dynamical regularity property for functions on the space $GX$.
\begin{defn} \label{def:bowen} 
Let $(X, \FFF)$ be a continuous flow.  A continuous function $\varphi$ on $X$ is said to have the \emph{Bowen property} if  there exists $V>0$ so that for any sufficiently small $\epsilon>0$,
\[ d(f_t(x), f_t(y))<\epsilon \mbox{ for all } t\in [0,S] \implies \left|\int_0^S \varphi(f_tx)dt - \int_0^S \varphi(f_ty)dt\right|<V \]
for any $x,y\in X$ and any $S>0$.
\end{defn}

We show that H\"older functions on $GX$ satisfy this property.

\begin{prop} \label{prop:bowen} 
If $\varphi$ is a H\"older continuous function on $GX$, then $\varphi$ satisfies the Bowen property for the geodesic flow $g_t$.
\end{prop}

\begin{proof}
We  prove that for any $V>0$, there exists an $\epsilon>0$ such that 
\[ 
d_{GX}(g_t(\gamma_1), g_t(\gamma_2))<\epsilon \mbox{ for all } t\in [0,S] \implies \left|\int_0^S \varphi(g_t\gamma_1)dt - \int_0^S \varphi(g_t\gamma_2)dt\right|<V 
\]
for any $\gamma_1, \gamma_2 \in GX$ and any $S>0$. The idea of the proof is that, using the $\CAT(-1)$ property for a comparison with $\mathbb{H}^2$, geodesics in $X$ which stay close over $[0,S]$ are in fact exponentially close over that range, from which the result follows.  The need to move between the metrics on $GX$ and $X$ adds some technicalities to the proof.

Let $V>0$ be given, and let $C, \alpha>0$ be the H\"older constants for $\varphi$ so that $|\varphi(\gamma_1,\gamma_2)|<Cd_{GX}(\gamma_1,\gamma_2)^\alpha$. We fix $\epsilon >0$ to be specified later. Suppose that $d_{GX}(g_t\gamma_1, g_t\gamma_2)<\epsilon$  for $t\in [0,S]$. By Lemma \ref{lem:tool2}, $d_X(\gamma_1(t), \gamma_2(t))< 2\epsilon$ for $t\in [0,S]$. By Lemma \ref{lem:lift track}, assuming that $2\epsilon<\epsilon_0$, lifting to the universal cover,  we have $d_{\tilde X}(\tilde \gamma_1(t), \tilde \gamma_2(t))<2\epsilon$ for $t\in [0,S]$. 

We construct a comparison pair of geodesic segments $c_1(t), c_2(t)$ in $\mathbb{H}^2$ with lengths $S$ and with distance at most $2\epsilon$ between their endpoints using the pair of triangles shown in Figure \ref{fig:quad}. By convexity of the distance function, $d_{\mathbb{H}^2}(c_1(t),c_2(t))<2\epsilon$. We translate the time parameter for $c_2$ by a constant $r$ so that at the point of their nearest approach in $\mathbb{H}^2$, both have the same time parameter. By interchanging the roles of $c_1$ and $c_2$ if necessary, we can assume that $r \geq 0$. We write $S':=S-r$.  Then, by a standard argument for the behavior of geodesics in $\mathbb{H}^2$, we have that
\[ d_{\mathbb{H}^2}(c_1(t), c_2(t+r))< 2\epsilon e^{-\min\{ t, S'-t \}} \mbox{ for all } t\in [0,S']. \]
Applying the $\CAT(-1)$ property, we have that
\[ d_{\tilde X}(\tilde \gamma_1(t), \tilde\gamma_2(t+r))< 2\epsilon e^{-\min\{ t, S'-t \}} \mbox{ for all } t\in [0,S'], \]
and we can push this estimate back down to $X$.

\begin{center}

\setlength{\unitlength}{.5pt}

\begin{picture}(600,250)(-300,-100)

\put(-230,100){$\tilde X$}
\put(230,100){$\mathbb{H}^2$}

\qbezier(-350,80)(-170,-20)(-50,80)	
\qbezier(-380,-80)(-200,20)(-80,-80)

\qbezier(-350,80)(-270,0)(-380,-80)
\qbezier(-50,80)(-150,0)(-80,-80)
\qbezier(-350,80)(-280,30)(-210,0)
\qbezier(-210,0)(-100,-60)(-80,-80)

\put(-270,50){$\tilde \gamma_1$}
\put(-300,-60){$\tilde \gamma_2$}

\put(-210,31){\circle*{5}}
\put(-210,-31){\circle*{5}}
\put(-220,4){\circle*{5}}

\qbezier(-210,31)(-215,17)(-220,4)
\qbezier(-220,4)(-215,-13)(-210,-31)

\put(-210,41){$p_1$}
\put(-235,-6){$q$}
\put(-210,-48){$p_2$}

\qbezier(80,80)(200,20)(380,80)
\qbezier(50,-80)(170,-20)(350,-80)
\qbezier(80,80)(80,0)(50,-80)
\qbezier(380,80)(350,0)(350,-80)

\qbezier(80,80)(210,0)(350,-80)

\put(210,50){\circle*{5}}
\put(210,-51){\circle*{5}}
\put(210,2){\circle*{5}}

\put(210,61){$\bar p_1$}
\put(190,-15){$\bar q$}
\put(210,-68){$\bar p_2$}

\qbezier(210,50)(210,0)(210,-51)

\put(140,65){$c_1$}
\put(140,-70){$c_2$}

\multiput(80,80)(0,-20){8}{\line(0,-1){10}}
\multiput(350,-80)(0,20){8}{\line(0,1){10}}

\put(80,-68){\circle*{5}}
\put(70,-95){$c_2(r)$}

\put(350,71){\circle*{5}}
\put(318,85){$c_1(S')$}

\end{picture}

\begin{figure}[h]
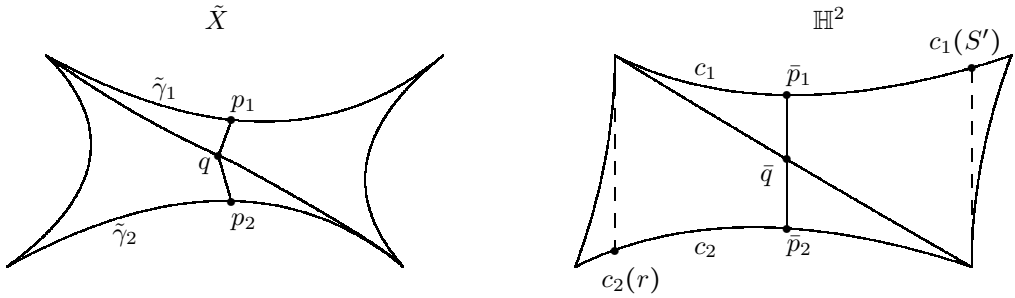

\caption{Comparison quadrilateral for Proposition \ref{prop:bowen}. Corresponding sides in the two quadrilaterals have the same length. By the $\CAT(-1)$ condition, $d_{\tilde X}(p_1,p_2)\leq d_{\mathbb{H}^2}(\bar p_1,\bar p_2)$.}\label{fig:quad}
\end{figure}

\end{center}

Next, using Lemma \ref{lem:shadow in X} we see that that there is a constant $T=T(4\epsilon)$ such that
\[d_{GX}(g_t\gamma_1, g_{t+r}\gamma_2) < 2 d_X(\gamma_1(t),\gamma_2(t+r) )< 4\epsilon e^{-\min\{ t, S'-t \}}\mbox{ for all } t\in [T, S'-T].\]
We recall from Lemma \ref{lem:shadow in X} that for small $\epsilon$, we can take $T(4\epsilon) = -\log(4\epsilon)$, and thus  $\lim_{\epsilon \to 0} \epsilon^\alpha T(4\epsilon) = 0.$ We assume $\epsilon$ is so small that $2C (3\epsilon)^\alpha T<V/3$.  

To control $|\int_0^S \varphi(g_t\gamma_1) dt - \int_0^S \varphi(g_t\gamma_2) dt|$, we first note that
\[
\left|\int_0^S \varphi(g_t\gamma_1) dt - \int_0^S \varphi(g_t\gamma_2) dt\right| \leq \left|\int_0^{S'} \varphi(g_t\gamma_1) dt - \int_r^S \varphi(g_t\gamma_2) dt \right| +2r\| \varphi \|.
\]
Since the flow is unit speed, $r \leq 2\epsilon$, and therefore, choosing $\epsilon$ so small that $4\epsilon \|\varphi \|<V/3$, and writing $\gamma_2' = g_r \gamma_2$, it suffices to control $|\int_0^{S'} \varphi(g_t\gamma_1) dt - \int_0^{S'} \varphi(g_{t}\gamma_2') dt|$.

We cover $[0,S']$ by the intervals $I_1=[0,T], I_2=(T, S'-T)$, and $I_3=[S'-T, S']$. Note that $I_2$ may be empty and $I_1$ and $I_3$ may overlap, depending on the values of $S'$ and $\epsilon$. Then,
\begin{align}
\Big| \int_0^{S'} \varphi(g_t\gamma_1) dt - \int_0^{S'} &\varphi(g_t\gamma'_2) dt \, \, \Big| \leq \int_0^{S'}|\varphi(g_t\gamma_1)-\varphi(g_t\gamma'_2)|dt \nonumber \\
&\leq \int_{I_1} |\varphi(g_t\gamma_1)-\varphi(g_t\gamma'_2)|dt + \int_{I_3} |\varphi(g_t\gamma_1)-\varphi(g_t\gamma'_2)|dt \nonumber \\
& + \int_{I_2} |\varphi(g_t\gamma_1)-\varphi(g_t\gamma'_2)|dt \nonumber.
\end{align}

Over $I_1$ and $I_3$, $d_{GX}(g_t\gamma_1, g_t\gamma'_2)< d_{GX}(g_t\gamma_1, g_t\gamma_2)+ d_{GX}(g_t\gamma_2, g_t\gamma'_2)<\epsilon+2\epsilon$, so by the H\"older condition, $|\varphi(g_t\gamma_1)-\varphi(g_t\gamma'_2)|\leq C (3\epsilon)^\alpha$. Thus
\[ \int_{I_1} |\varphi(g_t\gamma_1)-\varphi(g_t\gamma'_2)|dt + \int_{I_3} |\varphi(g_t\gamma_1)-\varphi(g_t\gamma'_2)|dt < 2C(3\epsilon)^\alpha T< V/3.\]

To bound the integral over $I_2$, we use the H\"older property again to obtain
\begin{align}
	\int_{I_2} |\varphi(g_t\gamma_1)-\varphi(g_t\gamma'_2)|dt &< \int_{I_2} Cd_{GX}(g_t\gamma_1, g_t \gamma'_2)^\alpha dt \nonumber \\
	& < \int_{I_2} C 4^\alpha \epsilon^\alpha e^{-\alpha \min\{t,S-t\}}dt \nonumber \\
	& < \epsilon^\alpha \int_0^\infty C 4^\alpha e^{-\alpha \min\{t,S-t\}}dt \nonumber<V/3,
\end{align}
where the last inequality comes from making a sufficiently small choice of $\epsilon$. Thus, $|\int_0^{S'} \varphi(g_t\gamma_1) dt - \int_0^{S'} \varphi(g_{t}\gamma_2') dt|<2V/3$, and so $|\int_0^S \varphi(g_t\gamma_1) dt - \int_0^S \varphi(g_t\gamma_2) dt|<V$.
\end{proof}

%

\subsection{Orbit closing lemma} \label{closing}

We prove a closing lemma for our setting, which gives what we call the \emph{weak periodic orbit closing property}. The idea is that for the suspension flow over a shift of finite type, an orbit segment can always be approximated by a periodic orbit. We show that this property passes to $GX$ using the orbit semi-equivalence. For a flow $(X, \FFF)$, we write $\Per(t)$ for the set of closed orbits of least period at most $t$.
\begin{defn} \label{orbitclose}
A continuous flow $(X, \FFF)$ satisfies the \emph{weak periodic orbit closing property} if for all $\epsilon>0$, there exists $R>0$ so that for any orbit segment $(\gamma, t)$, there exists $\gamma^\ast \in \Per(t+R)$ so that $d_t(\gamma, \gamma^\ast)<\epsilon$.
\end{defn}
\begin{lem} \label{lem:closing}
The geodesic flow on a compact $\CAT(-1)$ space satisfies the weak periodic orbit closing property.
\end{lem}
\begin{proof}
The proof uses many of the same ideas as the proof of Theorem \ref{thm:spec}. Let $\epsilon>0$ be given and fix an orbit segment $(\gamma, t)$ for $(GX, \{g_t\})$.  Let $h: \Susp(\Sigma,\sigma) \to GX$ be the orbit semi-equivalence provided by Proposition \ref{symbolic-coding}, where $\Sigma$ is a topologically transitive shift of finite type.  Let $T=T(\epsilon)$  be the constant from Lemma \ref{lem:shadow in X} and let $\delta>0$ satisfy that $y_1, y_2 \in \Susp(\Sigma,\sigma)$,  $d(y_1,y_2)<\delta$ implies $d_{GX}(h(y_1), h(y_2))<\epsilon/6$.

Fix a lift $(y, \hat t)$ under $h$ of $(g_{-T} \gamma, t+2\epsilon +2T)$, so
\[
\{h(\phi_s y): s \in [0, \hat t]\} = \{g_s \gamma: s \in[-T, t+T+ 2 \epsilon]\},
\]
where $\{\phi_s\}$ is the suspension flow. On $\Susp(\Sigma,\sigma)$, it is easy to check that we can close orbit segments to periodic orbits. That is, for all $\delta>0$, there exists $\hat R$ so that for all $(y, \hat t)$, there exists $y'$ so that $d_t(y, y')<\delta$ and $y'$ is periodic with period at most $\hat t + \hat R$. This property follows from the corresponding fact for $\Sigma$. We take such a point $y'$ for the orbit segment $(y, t)$ and $\delta>0$ under consideration. Then for all $s \in [0, \hat t]$, $d_{GX}(h(\phi_s y'), h(\phi_s y)) < \epsilon/6$. Thus,  writing $\gamma':=h(y')$ and reparameterizing, we see there is a time change $\rho$ so that for all $s \in [0, t + 2 \epsilon+2T]$, 
\[
d_{GX}(g_{\rho(s)} \gamma', g_s (g_{-T} \gamma)) < \epsilon/6.
\]
Using Lemma \ref{lem:tool2}, we see that for all $s \in [0, t + 2 \epsilon+2T]$, 
\[
d_X(\gamma'(\rho(s)), g_{-T}\gamma(s)) <\epsilon/3.
\]
Now we apply Proposition \ref{prop:nbhd shadow} to obtain that for all $s \in [0, t+2T]$
\[
d_X(\gamma'(s), g_{-T}\gamma(s)) <\epsilon.
\]
Now we apply Lemma \ref{lem:shadow in X} to obtain that for all $s \in [T, t+T]$, 
\[ 
d_{GX}(g_s\gamma', g_s(g_{-T} \gamma)) <2\epsilon,
\]
and thus for all $s \in [0, t]$, $d_{GX}(g_s(g_T\gamma'), g_s(\gamma)) <2\epsilon.$ We let $\gamma^\ast = g_T\gamma'$, and we have shown that $d_t(\gamma^\ast, \gamma)<2 \epsilon$.

Now it is clear that $\gamma^\ast$ is a periodic orbit, so it only remains to show that its period is controlled. Let $t^\ast$ be the period of $\gamma^\ast$. We observe that the orbit segment $(g_t \gamma^\ast, t^\ast-t)$ is a subset of the image under $h$ of the orbit segment $(\phi_{\hat t}y', R')$. So we let $R$ be a value so that for all $y \in \Susp(\Sigma,\sigma)$, the image of an orbit segment $(y, R')$ under the orbit equivalence $h$ is contained in the orbit segment $(h(y), R)$. This is possible by Corollary \ref{boundedtransition}.  Thus, the period of $\gamma^\ast$ is at most $t+R$, so at scale $2 \epsilon$, we have verified the property that we need.
\end{proof}

%

\section{Expansive flows with weak specification}\label{sec:applications}

We now establish the results on thermodynamic formalism and large deviations for $\CAT(-1)$ geodesic flows given in Theorem \ref{thm:applications}. The results are proved for expansive flows with weak specification, and thus apply to geodesic flow on compact $\CAT(-1)$ spaces in light of Theorem \ref{thm:weak-specification}. We prove

\begin{thm}\label{thm:applicationsgeneral}
Let $(X, \FFF)$ be a continuous flow on a compact metric space that is expansive and satisfies the weak specification property. Let  $\varphi: X \to \R$ be a continuous function satisfying the Bowen property. Then
\begin{enumerate}
\item the potential function $\varphi$ has a unique equilibrium measure $\mu_\varphi$,
\item the equilibrium measure $\mu_\varphi$ satisfies the Gibbs property,
\item if $(X, \FFF)$ satisfies the weak periodic orbit closing property, then the $\varphi$-weighted periodic orbits for the flow equidistribute to $\mu_\varphi$,
\item the ergodic measures are entropy dense in the space of $\FFF$-invariant probability measures,
\item the measure $\mu_\varphi$ satisfies the Large Deviations Principle.
\end{enumerate}
\end{thm}
We address each one of these properties in turn in the following subsections.


\subsection{Unique equilibrium states and the Gibbs property}

We refer to Walters \cite{Walters} as a standard reference for equilibrium states in discrete-time, and the article by Bowen and Ruelle \cite{Bowen-Ruelle} for flows.  Given a potential function $\varphi$, we study the question of whether there is a unique invariant measure which maximizes the quantity $h_\mu + \int\varphi\,d\mu$, where $h_\mu$ is the measure-theoretic entropy. More precisely, given a flow $\mathcal F$ on a compact metric space $X$, and a continuous function $\varphi: X \to \mathbb R$ (called the \emph{potential}), we define the \emph{topological pressure} to be
\[
P(\varphi) = \sup \left\{h_\mu + \int\varphi\,d\mu \, \, \Big\mid \, \, \mu \text{ is an } \mathcal F \text{-invariant probability measure} \right\},
\]
and an \emph{equilibrium state} for $\varphi$ to be a measure achieving this supremum. An equilibrium state for the constant function $\varphi = 0$ is called a \emph{measure of maximal entropy}. Equivalently, $P(\varphi)$  is  the exponential growth rate of the number of distinct orbits for the system, weighted by $\varphi$ in the following sense. For an expansive flow, the precise definition is
\[
P(\varphi) = \lim_{t \to \infty} \frac{1}{t} \log \sup \left \{ \sum_{x \in E} e^{\int_0^t\varphi(g_sx)} \, \, \Big\mid \, \, E \text{ is a $(t, \epsilon)$-separated set} \right \},
\]
where $\epsilon$ is an expansivity constant for the flow, and a set $E$ is $(t, \epsilon)$-separated if for every distinct $x, y\in E$ we have $y \notin \overline B_t(x, \epsilon)$.

For a continuous function $\varphi: X \to \R$, an invariant measure $\mu$ has the \emph{Gibbs property for $\varphi$} if for all $\rho>0$, there is  a constant $Q =Q(\rho)>1$  such that for every $x\in X$ and $t \in \mathbb R$, we have
\begin{equation}\label{eqn:gibbs}
Q^{-1} e^{-tP(\varphi) + \Phi(x,t)} \leq \mu(B_t(x,\rho)) \leq Q e^{-tP(\varphi) + \Phi(x,t)},
\end{equation}
where $\Phi(x,t) = \int_0^t \varphi(f_sx)\,ds$ and $B_t(x, \rho) = \{y: d(f_sx, f_sy)<\rho \text{ for all }s \in [0,t]\}$. In particular, a measure  has the Gibbs property for the function $\varphi =0$ if for all $\rho>0$, there is  a constant $Q =Q(\rho)>1$  such that for every $x\in X$ and $t \in \mathbb R$, we have
\begin{equation}\label{eqn:gibbsmme}
Q^{-1} e^{-th} \leq \mu(B_t(x,\rho)) \leq Q e^{-th},
\end{equation}

For an expansive flow, there exists an equilibrium state for every continuous potential. However, uniqueness can be a subtle question. In our setting, we have the following statement.

\begin{thm}\label{BFCT}
Let $(X, \mathcal F)$ be a continuous flow on a compact metric space. Suppose that $\mathcal F$ is expansive and has the weak specification property. Then, for every potential $\varphi$ with the Bowen property, there exists a unique equilibrium state $\mu_\varphi$. Every such measure $\mu_\varphi$ satisfies the Gibbs property for $\varphi$.
\end{thm}

For flows with the strong version of specification, this result was proved by Franco \cite{eF77}, generalizing Bowen's  discrete-time argument \cite{bowen}. The same essential argument applies assuming only weak specification. However, non-trivial technical issues must be overcome since weak specification does not allow us to use periodic orbits in the construction of the unique equilibrium state, and there are additional technicalities in various counting arguments. Formally, the statement for weak specification is a corollary of recent work by Climenhaga and the third named author \cite{ClimenhagaThompsonflows}, although that work is designed to apply much more generally in settings which do not have any global form of the specification property.


\subsection{Equidistribution of weighted periodic orbits} \label{sec:equidistribute}

For $a<b$, let $\Per (a, b]$ denote the set of closed orbits for $\{f_s\}$ with period in the interval $(a, b]$, and let $\varphi$ be a continuous function. We define the \emph{upper pressure of periodic orbits} to be 
\begin{equation} \label{d.gurevicpressure}
\overline P^{\ast} (\varphi) = \limsup_{t \to \infty} \frac{1}{t} \log \sum_{\gamma \in \Per(t-R, t]} e^{\Phi(\gamma)},
\end{equation}
where $R>0$ is fixed and $\Phi(\gamma)$ is the value given by integrating $\varphi$ around the periodic orbit. For an expansive flow, $\overline P^\ast (\varphi)$ is well defined, and satisfies $\overline P^\ast (\varphi) \leq P(\varphi)$.  This was proved in the $\varphi =0$ case in \cite{BW}. To extend to $\varphi \neq 0$,  the proof of \cite[Theorem 5]{BW} shows that choosing one point $x_\gamma$ on each of the orbits $\gamma$ in $\Per(t-R, t]$ yields a $(t, \alpha)$-separated set for some small $\alpha>0$. Since $| \Phi(\gamma)-\int_0^t \varphi(g_sx_\gamma)| \leq R \sup|\varphi|$, it follows that $\overline P^\ast(\varphi)\leq P(\varphi)$. It is a straightforward exercise to verify that the value of $\overline P^\ast (\varphi)$ is independent of the choice of $R$. 

We define the \emph{lower pressure of periodic orbits (with window size $R$)} to be
\begin{equation} \label{d.gurevicpressurel}
\underline P^{\ast}_R (\varphi) = \liminf_{t \to \infty} \frac{1}{t} \log \sum_{\gamma \in \Per(t-R, t]} e^{\Phi(\gamma)}.
\end{equation}
If there exists $R$ such that $\underline P^{\ast}_R (\varphi) = \overline P^{\ast} (\varphi)$, then $\underline P^{\ast}_{R'} (\varphi) = \overline P^{\ast} (\varphi)$ for any $R' \geq R$, and  we call this common value the \emph{pressure of periodic orbits}, denoted $P^\ast(\varphi)$. 

For a periodic orbit  $\gamma$, let $\mu_\gamma$ be the natural measure around the orbit. That is, if $\gamma$ has period $t$, and $x \in \gamma$, then
\[
\int \psi d\mu_\gamma := \frac{1}{t} \int_0^t \psi (f_sx)ds
\]
for all $\psi \in C(X)$. We say the \emph{periodic orbits weighted by $\varphi$ equidistribute} to a measure $\mu$ if for any fixed $R>0$ which is sufficiently large, we have
\begin{equation}\label{eq.equidist}
\frac{1}{C(t, R)}  \sum_{\gamma \in \Per(t-R, t]}e^{\Phi(\gamma)} \mu_\gamma \to \mu,
\end{equation} 
where $C(t, R)$ is the normalizing constant $(\sum_{\gamma \in \Per(t-R,t ]} e^{\Phi(\gamma)} \mu_\gamma)(X)$.  Equidistribution of weighted periodic orbits for equilibrium states  was first investigated in a uniformly hyperbolic setting by Parry \cite{wP88}, and for geodesic flow on manifolds of non-positive curvature by Pollicott \cite{pollicott_nonpos}. 

The proof of the Variational Principle \cite[Theorem 9.10]{Walters} shows that if $\underline P^\ast_R(\varphi) =P(\varphi)$, then any weak$^\ast$ limit of $\frac{1}{C(t, R)}  \sum_{\gamma \in \Per(t-R, t]}e^{\Phi(\gamma)} \mu_\gamma$ is an equilibrium state for $\varphi$. See Remark 3 of \cite{GS14} and \S2.3 of \cite{BCFT}.
Thus if we know that $P^\ast(\varphi) =P(\varphi)$, and that $\varphi$ has a unique equilibrium state $\mu$, it follows immediately that the periodic orbits weighted by $\varphi$ equidistribute to $\mu$. 

\begin{lem} \label{lem:closing0}
Suppose an expansive flow $(X, \FFF)$ has the weak periodic orbit closing property of Definition \ref{orbitclose}. Then there exists $R>0$ so that for any continuous potential with the Bowen property, $\underline P^\ast_R(\varphi) =P(\varphi)$, and thus  $P^\ast(\varphi) =P(\varphi)$.
\end{lem}

\begin{proof}

We already verified that $\overline P^\ast(\varphi) \leq P(\varphi)$. For the other inequality, let $2\epsilon$ be an expansivity constant and take a sequence of $(t, 2\epsilon)$-separated sets $E_t$ so that
\[
\frac{1}{t} \log \sum_{x \in E_t}e^{\int_0^t\varphi(g_sx)} \to P(\varphi).
\]

Then by the weak periodic orbit closing property, for each $x \in E_t$, there exists a periodic orbit $\gamma(x)$ with $d_t(x, \gamma(x))<\epsilon$ and $\{\gamma(x) \mid x \in E_t\} \subset \Per(t, t+R]$. For any fixed $\gamma \in \Per(t, t+R]$, since $E_t$ is $(t, 2\epsilon)$-separated,  there are at most $(T+R)/2\epsilon$ elements in the set $\{ x \in E_t : \gamma(x)=\gamma \}$. We also have
\[
\left | \Phi (\gamma(x)) - \int_0^t\varphi(g_sx) \right | \leq \left | \int_0^t\varphi(g_s \gamma(x)) - \int_0^t\varphi(g_sx) \right | + R \| \varphi \| \leq V+ R \| \varphi \|,
\]
where $V$ is the constant appearing in the Bowen property for $\varphi$. Thus,
\[
\sum_{\gamma \in \Per(t, t+R]} e^{\Phi(\gamma)} \geq \sum_{\{\gamma(x) \mid x \in E_t\}}e^{\Phi(\gamma)}  \geq \frac{2\epsilon}{T+R}e^{-V-R \| \varphi \|}\sum_{x \in E_t}e^{\int_0^t\varphi(g_sx)},
\]
and so
\[
\frac{1}{t+R}\log \sum_{\gamma \in \Per(t, t+R]} e^{\Phi(\gamma)} \geq \frac{t}{t+R} \left( \frac{1}{t} \log \sum_{x \in E_t}e^{\int_0^t\varphi(g_sx)} \right)  - \frac{K}{t+R},
\]
where $K = V+R \| \varphi \|- \log(2\epsilon(T+R)^{-1})$. Taking a limit as $t \to \infty$, we obtain $\underline P^\ast_R(\varphi) \geq P(\varphi)$. We already verified that $\overline P^\ast(\varphi) \leq P(\varphi)$, so this completes the proof. 
\end{proof}

Thus, for an expansive flow with weak specification and weak periodic orbit closing, and any continuous $\varphi: X \to \mathbb R$ with the Bowen property, since 
$\varphi$ has a unique equilibrium state $\mu_\varphi$, it follows that the periodic orbits weighted by $\varphi$ are equidistributed in the sense that for any fixed sufficiently large $R>0$,
\[
\frac{1}{C(t, R)}  \sum_{\gamma \in \Per(t-R, t]}e^{\Phi(\gamma)} \mu_\gamma \to \mu_\varphi.
\]

We remark that a stronger equidistribution statement can be asked for by allowing $R>0$ to be ANY fixed window size in the above. This stronger version is what is obtained in the setting of e.g. \cite{wP88, BCFT}. We emphasize that this stronger statement cannot be obtained from our hypotheses because knowledge of $\underline P^\ast_R(\varphi)$ a priori gives no information on $\underline P^\ast_{\delta} (\varphi)$ for $\delta<R$, and the weak specification and periodic orbit closing hypotheses are not strong enough to ensure that there are periodic orbits of length $[T, T+\delta)$ when $\delta$ is small.
\subsection{Entropy density of ergodic measures} \label{sec:hdense}
For a discrete-time dynamical system $(X, f)$ or flow $(X, \FFF)$, the \emph{entropy density of ergodic measures} is the property that for any invariant measure $\mu$, for any $\eta >0$, we can find an ergodic measure $\nu$ such that $D(\mu, \nu)< \eta$ and $| h_\nu - h_\mu| < \eta$, where $D$ is any choice of metric on the space of measures on $X$ compatible with the weak$^\ast$ topology (see  \S6.1 of \cite{Walters}). 

Entropy density is known to be true for maps with the almost product property  \cite{PfS}, which is a weaker hypothesis than the specification property. The basic argument was first proved for $\Z^d$-shifts with specification by Eizenberg, Kifer and Weiss \cite{EKW}. No reference is available for maps with weak specification, or for flows.  In this section, we carefully prove entropy density for flows with weak specification. While this extension is expected, care must be taken in the argument, and dealing with the variable gap length is a non-trivial extension of the existing proofs.

We remark that the time-$1$ map $f_1$ of a flow with weak specification may not satisfy the entropy density condition.  Consider a suspension flow with constant roof function $1$. An ergodic measure for $f_1$ is supported on a single height, i.e  on $X \times \{h\}$ for some $h \in[0,1)$. Take an $f_1$-invariant measure given by a convex combination of an ergodic measure on $X \times \{0\}$, and an ergodic measure on $X \times \{ \frac1 2\}$. This measure can clearly not be approximated weak$^\ast$ by an ergodic $f_1$-invariant measure.

We remark that entropy density of ergodic measures is not true for geodesic flow on many $\CAT(0)$ spaces.  The ergodic measures are not even dense. For example, we can take the setting of Theorem \ref{thm:obstruction} and consider a $\CAT(0)$ space with a fat 1-flat. A measure whose support is two distinct parallel geodesics in the flat is not a weak$^\ast$ limit of ergodic measures. This phenomenon was proved rigorously in \cite{CS} for rank one surfaces with an embedded flat cylinder.

Before we proceed, we first require a general lemma that says that weak specification actually allows us to approximate infinitely many orbit segments.

\begin{lem}\label{infinitary-specification} 
Let $(X, \mathcal F)$ be a continuous flow on a compact metric space and assume that $\mathcal F$ satisfies the 
weak specification property. Then the conclusion of the specification property holds for any countably 
infinite sequence of orbit segments.
\end{lem}

\begin{proof}
Let $\delta >0$ be the scale, and $\tau>0$ the maximum transition time for the scale $\delta/3$ provided 
by the weak specification property for $\mathcal F$. Let $\{(x_i,t_i)\}_{i\in \mathbb N}$ be a countably 
infinite sequence of orbit segments. For each $j\in \mathbb N$, we use the weak specification on the first $j$ orbit 
segments $\{(x_i,t_i)\}_{i=1}^j$ to produce a point $y_j\in X$ and corresponding
transition times $\tau^{(j)}_i$ ($1\leq i\leq j$), so that appropriate iterates of $y_j$ $(\delta/3)$-shadow the prescribed orbit segments. 
Since the space $X$ is compact, one can choose an accumulation point
for the sequence $\{y_j\}_{j\in \mathbb N}$, call it $y$. Passing to a subsequence, we may assume that
$y_j \rightarrow y$.

We now want to verify that $y$ has the desired property. To do this, we need to produce a countable collection
$\tau_i$ of transition times, and check the corresponding specification property. First,
look at the sequence $\{\tau^{(j)}_1\}_{j\in \mathbb N} \subset [0, \tau]$. Passing to a subsequence if necessary,
we may assume $\{\tau^{(j)}_1\}_{j\in \mathbb N}$ converges to $\tau_1 \in [0, \tau]$. Next consider the 
sequence $\{\tau^{(j)}_2\}_{j\geq 2, j\in \mathbb N} \subset [0, \tau]$. Again, passing to a subsequence, we 
can choose a limiting $\tau_2\in [0, \tau]$. Continuing in this manner, we obtain a sequence of transition times
$\{\tau_i\}_{i\in \mathbb N}$. 

Now, given $k \in \mathbb N$,
we consider the finitely many orbit segments $\{(x_i, t_i)\}_{i=1}^k$. Recall that $s_j := \sum_{i=1}^{j} n_i + \sum_{i=1}^{j-1}\tau_i$ is the time taken to shadow the first $j$ orbit segments. By compactness, there is an
$\epsilon >0$ with the property that, for any pair of points satisfying $d(z, z') \leq \epsilon$, we have
$d_{s_k}(z, z')<\delta/3$. By continuity of the flow, there is also an $\epsilon' >0$ so that
for all $x\in X$, $|t-t'|< \epsilon '$, and $1\leq i \leq k$, we have $d_{t_i}(f_t(x), f_{t'}(x))<\delta/3$.
We now choose a $y':= y_N$ from the approximating sequence having the following two properties:
(i) $d(y', y) < \epsilon$, and (ii) each $|\tau^{(N)}_i - \tau_i|< \epsilon'/k$, for $1\leq i \leq k$. 

From property (i), we conclude that $d_{s_k}(y, y')<\delta/3$, and from property (ii), it follows immediately 
that $|(s^{(N)}_i+ \tau^{(N)}_i) - (s_i+\tau_i)| < \epsilon'$ holds for all $1\leq i \leq k$. 
We now have the estimate:
\begin{align*}
d_{t_i}(f_{s_{i-1} + \tau_{i-1}}y, x_i) &\leq d_{t_i}(f_{s_{i-1} + \tau_{i-1}}y, f_{s_{i-1} + \tau_{i-1}}y') + 
d_{t_i}(f_{s_{i-1} + \tau_{i-1}}y', x_i) \\
& \leq d_{s_k}(y, y') + d_{t_i}(f_{s_{i-1} + \tau_{i-1}}y', x_i) \\
& \leq d_{s_k}(y, y') + d_{t_i}(f_{s_{i-1} + \tau_{i-1}}y', f_{s^{(N)}_{i-1} + \tau^{(N)}_{i-1}}y') + d_{t_i}(f_{s^{(N)}_{i-1} + \tau^{(N)}_{i-1}}y', x_i) \\
& \leq \delta/3 + \delta/3 + \delta /3 = \delta.
\end{align*}
The first and third inequalities are just applications of the triangle inequality for the metric $d_{t_i}$. 
The second inequality comes from the definition of the metrics $d_t$, along with the fact that 
$s_{i-1} + \tau_{i-1}+ t_i \leq s_k$ for every $1\leq i \leq k$.  For the last inequality, the first term is controlled
by property (i), while the second term is controlled by property (ii) and the choice of $\epsilon'$. The last
term is controlled by the specification property at scale $\delta/3$ for the point $y'=y_N$. This gives the
desired estimate, and since this can be done for every $k\in \mathbb N$, completes the proof.
\end{proof}

Let $\mathcal M_{\FFF}(X)$ denote the space of $\FFF$-invariant probability measures on $X$. The following proposition is the main result of this section. 

\begin{prop} \label{entropydenseflow}
Let $\FFF$ be an expansive flow with the weak specification property. Then the ergodic measures are entropy dense in $M_{\FFF}(X)$. That is,  if $\mu \in \mathcal M_{\FFF}(X)$, then for any $\eta >0$ we can find an $\FFF$-invariant ergodic measure $\nu$ such that $D(\mu, \nu)< \eta$ and $| h_\nu - h_\mu| < \eta$.
\end{prop}
The strategy is to construct a closed $\FFF$-invariant set $Y\subset X$ such that every invariant measure supported on $Y$ 
is weak*-close to $\mu$, and such that the topological entropy of $Y$ is close to $h_\mu$. For $x \in X$ and $t \in \mathbb R$, we define a measure $\mathcal{E}_t(x)$ by
\[
\int \psi\,d\EEE_t(x) = \frac 1t \int_0^t \psi(f_s x)\,ds,
\]
for all $\psi\in C(X)$. The measures $\EEE_t(x)$ are sometimes called the \emph{empirical measures} for the flow; they are not $\FFF$-invariant in general. Given a set  $U\subset \mathcal M(X)$, let
\[ 
X_{t, U} := \{x \in X \mid\EEE_t(x) \in U\}.
\]

From now on, we fix $\eta>0$, and let $ \mathcal B:= B(\mu, 5 \eta)$ and for $m\geq 1$, let
\begin{equation}\label{eqn:Y}
Y_m := \{x \mid f_{s}x \in  X_{m,  \mathcal {\overline B}} \text{ for all }s \geq 0 \}.
\end{equation}
Each $Y_m$ is closed and forward invariant, so we can consider the dynamics of the semi-flow $\FFF^+=\{f_t: t \geq0\}$ on $Y_m$. We could modify the definition of $Y_m$ by replacing ``$s\geq 0$'' with ``$s \in \R$'' to get a flow-invariant set, but we avoid this to simplify the book-keeping of arguments that appear later in our proof. It is unproblematic to work with a set which is only forward invariant because measures which are invariant for $\FFF^+|_{Y_m}$ can easily be shown to be invariant for $\mathcal F$. More precisely, consider $\nu\in\mathcal{M}_{\FFF^+}(Y_m)$. Then for each $t \geq0$, $\nu\in \mathcal M_{f_t}(X)$. Since $f_t$ is invertible, then $\nu$ is $f_{-t}$ invariant. Thus $\nu \in \mathcal M_{\FFF}(X)$. We prove the following lemma.
\begin{lem}\label{ergodic}
For any $m \geq 1$, if $\nu\in\mathcal{M}_{\FFF^+}(Y_m)$, then $D(\mu,\nu)\le 6 \eta$.

\begin{proof}
Assume that $\nu\in\mathcal{M}_{\mathcal F^+}(Y_m)$ is ergodic. Then there exists a generic point $x\in Y_{m}$ so
$\mathcal{E}_t(x)$ converges to $\nu$. For a large value of $t$, we chop the orbit $(x,t)$ into segments of length $m$ and a remainder, and use that for each $i$, $f_{im}x \in  X_{m,  \mathcal {\overline B}}$. More precisely, for $t\in \R$, write $t=sm+q$ where $s$ is an integer and $0\le q<m$. Then
\begin{align*}
D(\mathcal{E}_t(x),\mu)
&\le
\sum_{i=0}^{s-1}\frac{m}{t}D(\mathcal{E}_{m}(f_{im}x),\mu)
+\frac{q}{t}D(\mathcal{E}_q(f_{sm}x),\mu).
\end{align*}
Since by \eqref{eqn:Y}, $D(\mathcal{E}_{m}(f_{im}x),\mu)\leq 5\eta$, we have $\sum_{i=0}^{s-1}\frac{m}{t}D(\mathcal{E}_{m}(f_{im}x),\mu) \leq 5 \eta$. For the remaining error term, writing $M$ for the diameter of the space of probability measures on $X$, let $t$ be large enough so that $mM/t< \eta$. Then $D(\mathcal{E}_t(x),\mu) < 6 \eta$. Thus, taking $t\rightarrow\infty$, we have the lemma for $\nu$ ergodic. The result for $\nu$ non-ergodic follows from ergodic decomposition.
\end{proof}
\end{lem}

We will let $Y:= Y_{Kn}$ for values of $K$ and $n$ to be chosen shortly. By expansivity, the entropy map $\mu \to h_\mu$ is upper semi-continuous. So by the variational principle and the fact that measures in $Y$ are weak$^\ast$-close to $\mu$, then the topological entropy of $Y$ cannot be much larger than $h_{\mu}$; by choosing $\eta$ small enough, we can guarantee that $h(Y) < h_{\mu}+\gamma$. To show that $Y$ has entropy close to $h_{\mu}$, we use our specification property to build a large number of $(t,\epsilon)$-separated points inside $Y$ for arbitrarily large $t$, thus giving a lower bound on the topological entropy of $Y$.

We rely on the following result, whose proof is a general argument based on the definition of entropy and the Birkhoff ergodic theorem. In the discrete-time case, it is a corollary of Proposition 2.1 of \cite{PfS} (see also Proposition 2.5 of \cite{Yamamoto}).

\begin{prop}\label{lem:nepssep}
Let $\mu$ be ergodic and $h< h_{\mu}$. Then there exists $\epsilon>0$ such that for any neighborhood $U$ of $\mu$, there exists $T$ so that for any $t \geq T$ there exists a $(t, \epsilon)$-separated set $\Gamma \subset X_{t, U}$ such that $\# \Gamma \geq e^{th}$.
\end{prop}
Now use the ergodic decomposition of $\mu$ to find $\lambda = \sum_{i=1}^p a_i \mu_i$ such that the $\mu_i$ are ergodic, the $a_i \in (0,1)$ such that $\sum_{i=1}^p a_i=1$, $D(\mu,\lambda)\leq \eta$, and $h_\lambda > h_{\mu}-\eta$.
See \cite{Yo} for a proof that this is possible.

Let $h_i = 0$ when $h_{\mu_i}=0$, and $\max(0,\, h_{\mu_i} - \eta) < h_i < h_{\mu_i}$ otherwise.  Take $3 \epsilon_i$ and $T_i$ so that  the conclusion of Proposition \ref{lem:nepssep} holds for $\mu_i$ and $h_i$, and let $\epsilon'$ be the minimum of the $\epsilon_i$, and $T$ be the maximum of the $T_i$.  Let 
\[
\Var(D, \epsilon):= \sup \{D(\delta_x, \delta_y) \mid d(x, y)< \epsilon\},
\]
where $\delta_x$ denotes the Dirac measure at $x$. Note that since the map $x \to \delta_x$ is continuous, we have $\Var(D, \epsilon) \to 0$ as $\epsilon \to 0$. Choose $\epsilon<\epsilon'$ so that $\Var(D, \epsilon) < \eta$. Choose $t$ such that letting $t_i := a_it $, then $t_i\geq T$ for every $i$. Note that $t = \sum_{i=1}^p t_i$. We are free to choose $t$ as large as we like relative to $p$, and $\tau(\epsilon)$, the maximum transition time provided by the weak specification property for $\FFF$ at scale $\epsilon$. We will specify how large $t$ should be chosen later. 

Let $U_i = B(\mu_i, \eta)$. Take  $(t_i, 3 \epsilon)$-separated sets $\Gamma_i \subset X_{t_i, U_i}$ such that $\# \Gamma_i \geq e^{t_ih_i}$. Now we use the weak specification property for the flow at scale $\epsilon$ to define a map
\[
\Phi: \prod_{i=1}^\infty (\Gamma_1 \times \cdots  \times \Gamma_p) \to X.
\]
That is, given $(x_{11}, \ldots x_{1p}, x_{21}, \ldots, x_{2p}, \ldots)$, where $x_{ij} \in \Gamma_j$, we find a point $y \in X$ which $\epsilon$-shadows $(x_{11}, t_1)$, then after a transition period of time at most $\tau$, $\epsilon$-shadows $(x_{12}, t_2)$, and so on. Such a $y$ can be found by the infinitary version of the weak specification property, see Lemma \ref{infinitary-specification}.

We will show that for sufficiently large $t$, the image of $\Phi$ is a subset of $Y$. We then use $\Phi$ to construct $(t, \epsilon)$-separated sets for large $t$ which satisfy cardinality estimates that yield the estimate we require on $h(Y)$.

First we show that the image of $\Phi$ belongs to $Y$. The construction was chosen so that each time a portion of the orbit of $y$ approximates a sequence of orbit segments in $\Gamma_1 \times \cdots  \times \Gamma_p$, the orbit has spent exactly the right amount of time approximating each of $\mu_1, \ldots, \mu_p$ so that the appropriate empirical measure for $y$ is close to $\mu$. Thus, in what follows, we show that the empirical measures of $y$ are close to $\mu$ along a subsequence corresponding to the times when $y$ approximates a sequence in $\prod_{i=1}^k (\Gamma_1 \times \cdots  \times \Gamma_p)$. From there we bootstrap to all sufficiently large times.   

Fix a point $y$ in the image of $\Phi$, so $y=\Phi (x_{11}, \ldots x_{1p}, x_{21}, \ldots, x_{2p}, \ldots)$, where $x_{ij} \in \Gamma_j$ for all $i\geq1, j\in\{1, \ldots, p\}$. Let $\tau_{ij}(y)$ be the length of the transition time in the specification property that occurs immediately after approximating the orbit segment $(x_{ij}, t_j)$.  
Let $c = \sum_{i=1}^p t_i + (p-1)\tau$ and $b_k=kc+(k-1)\tau$. Then $c$ is the upper bound on the total time taken to approximate a sequence of orbits in $\Gamma_1 \times \cdots  \times \Gamma_p$, and $b_k$ is the upper bound on time spent approximating a sequence of orbits in $\prod_{i=1}^k(\Gamma_1 \times \cdots  \times \Gamma_p)$. The precise time to approximate such a sequence of orbits for a point $y$ is given by $c_k(y) = \sum_{i=1}^p t_i + \sum_{i=1}^{p-1}\tau_{ki}(y)$ and 
$b_k(y)= \sum_{i=1}^k c_i(y) +\sum_{i=1}^{k-1}\tau_{ip}(y)$ respectively (with $b_0=b_0(y)=0$).

\begin{lem}\label{lem:approx}
If $t$ was chosen sufficiently large, then $D(\EEE_{c}(f_{b_k(y)}y),\mu)\leq 5\eta$ for all $k \geq 0$.
\end{lem}

\begin{proof}
Fix $k\geq1$, and write $y'=f_{b_{k-1}(y)}y$, $\tau_j= \tau_{kj}(y)$, and $s_i = \sum_{j=1}^it_j + \sum_{j=1}^{i-1} \tau_j$, so $s_i$ is the total time that $y'$ initially spends approximating the corresponding sequence in $\Gamma_1 \times \cdots  \times \Gamma_i$. Then, writing $M$ for the diameter of $\mathcal M_{\mathcal F}(X)$ in the metric $D$, we remove the `uncontrolled' portion of the orbit of $y$ from consideration by using the estimate
\[
D \left (\EEE_{c}(y'), \sum_{i=1}^p \frac{t_i}{c} \EEE_{t_i}(f_{s_{i-1}+ \tau_{i-1}}y') \right) \leq \frac{p}{c} \tau M.
\]
Now since $d_{t_i}(f_{s_{i-1}+ \tau_{i-1}}y', x_{ki})< \epsilon$, for each $i$, we have
\[
D \left (\EEE_{t_i}(f_{s_{i-1}+ \tau_{i-1}}y'), \EEE_{t_i}(x_{ki}) \right)< t_i \Var(D, \epsilon) < t_i \eta.
\]
Thus, by choosing $t$, and hence $c$, so large that $\frac{p}{c} \tau M< \eta$, we have
\[
D \left (\EEE_{c}(y'), \sum_{i=1}^p \frac{t_i}{c} \EEE_{t_i}(x_{ki}) \right) < \frac{p}{c} \tau M + \sum_{i=1}^p\frac{t_i}c \eta  < 2 \eta.
\]
Now since for each $i$, $x_{ki} \in X_{t_i, U_i}$, we have
\[
D \left (\sum_{i=1}^p \frac{t_i}{c} \EEE_{t_i}(x_{ki}), \sum_{i=1}^p \frac{t_i}{c} \mu_i \right) \leq \sum_{i=1}^p \frac{t_i}{c}\eta < \eta.
\]
Furthermore, we have  $t \leq c = \sum_{i=1}^p t_i + (p-1)\tau \leq t+p\tau$, so if $t$ is chosen to be much larger than $p\tau$ then $t_i/c$ is close to $t_i/t=a_i$ and we can ensure that
\[
D \left (\sum_{i=1}^p \frac{t_i}{c} \EEE_{t_i}(x_{ki}), \sum_{i=1}^p a_i \mu_i \right) < \eta. \]
Putting all this together, we have 
\begin{align*}
D(\EEE_{c}(f_{b_{k-1}}y),\mu) \leq & ~D \left (\EEE_{c}(y'), \sum_{i=1}^p \frac{t_i}{c} \EEE_{n_i}(x_{ki}) \right)
+ D \left (\sum_{i=1}^p \frac{t_i}{c} \EEE_{t_i}(x_{ki}), \sum_{i=1}^p \frac{t_i}{c} \mu_i \right) \\
& + D \left (\sum_{i=1}^p \frac{t_i}{c} \mu_i, \sum_{i=1}^p a_i \mu_i \right) 
+ D\left (\sum_{i=1}^p a_i \mu_i, \mu \right) < 5\eta. \qedhere
\end{align*}
\end{proof}
The previous lemma was where we required that $t$ is large relative to $\tau$ and $p$. In the next lemma, we specify how large $K$ needs to be chosen. The idea is that an orbit segment of $y$ of length $K(c+ \tau)$ will consist of $K-2$ sub-segments of length $c$ where Lemma \ref{lem:approx} applies and so the empirical measures along the subsegments are close to $\mu$. Additional deviation of the empirical measure along the whole orbit segment is made arbitrarily small  by choosing $K$ large. This is the strategy for the proof of the following lemma.

\begin{lem} \label{lem: in Y}
If $y$ is a point in the image of $\Phi$, then $y \in Y$.
\end{lem}

\begin{proof}
Given $s\geq0$, we need to show that $f_sy \in X_{Kt, \overline {\mathcal B}}$ for a suitably chosen $K$. The idea is that taking the unique $m$ so that $b_{m}(y) < s \leq b_{m+1}(y)$, we have
\[
\EEE_{Kt}(f_s y) = \sum_{i=1}^{K-2}\frac{c}{Kt} \EEE_c (f_{b_{m+i}(y)}y) + \text{ error. }
\]
The error term has two sources. First, there are at most $K$ segments of $y$'s orbit, each of length at most $\tau$, used as the transition segments in the application of the specification property in the construction of $\Phi$. Second, there is a run of length at most $t+\tau$ at both the start and end of the orbit segment $(f_sy, Kt)$. More precisely, using Lemma \ref{lem:approx}, we have  
\[
D(\EEE_{Kt}(f_sy), \mu) \leq \frac{c(K-2)}{Kt} 5 \eta + \frac{\tau K}{Kt}M + \frac{2M(t+\tau)}{Kt} \leq \frac{c}{t} 5 \eta +\frac{\tau M}{t} + \frac{2M}{K} + \frac{2M\tau}{Kt}.
\]
We see that if $K$ and $t$ are large enough, then the right hand side is arbitrarily small. Thus $y \in Y_{Kt}=Y$.
\end{proof}

Now we prove our entropy estimates. 
We use $\Phi$ to define a map
\[
\Phi_m: \prod_{i=1}^m (\Gamma_1 \times \cdots  \times \Gamma_p) \to Y.
\]
For each $\underline x \in \prod_{i=1}^m (\Gamma_1 \times \cdots  \times \Gamma_p)$, we make a choice of $\underline y \in \prod_{i=1}^\infty (\Gamma_1 \times \cdots  \times \Gamma_p) $ with $y_{ij}=x_{ij}$ for $i \in \{1, \ldots, m\}$, $j \in \{1, \ldots, p\}$, and we define $\Phi_m(\underline x):= \Phi(\underline y)$. By Lemma \ref{lem: in Y}, the image of $\Phi_m$ belongs to $Y$. For $j \in\{1, \ldots, mp-1\}$, let $\tau_j(\underline x) \in [0, \tau]$ denote the $j$th transition time that occurs when applying the specification property in the definition of $\Phi_m(\underline x)$.

\begin{lem}
There exists a constant $C$ so that for all $m$, the image of $\Phi_m$ contains a $(b_m, \epsilon/2)$-separated set $E_m$ with $\# E_m \geq C^{-m}\# \prod_{i=1}^m (\Gamma_1 \times \cdots  \times \Gamma_p)$.
\end{lem}

\begin{proof}
Let $k\in \mathbb N$ be large enough so that, writing $\zeta:=\tau/k$,  we have $d(x,f_sx)<\epsilon/2$ for every $x\in X$ and $s\in (-\zeta,\zeta)$. We partition the interval $[0,mp \tau]$ into $kmp$ sub-intervals $I_1,\dots, I_{kmp}$ of length $\zeta$, denoting this partition as $P$.  

Given $\underline x \in \prod_{i=1}^m (\Gamma_1 \times \cdots  \times \Gamma_p)$, take the sequence $n_1, \ldots, n_k$ so that
\[
\tau_1(\underline x) + \cdots + \tau_i(\underline x) \in I_{n_i} \text{ for every } 1\leq i \leq mp-1.
\]
Now let $l_1=n_1$ and $l_{i+1}=n_{i+1}-n_i$ for  $1\leq i \leq k-2$, and let $l(\underline x):= (l_1,\dots,l_{k-1})$. Since $\tau_{i+1}(\underline x)\in[0, \tau]$, we have $n_i \leq n_{i+1} \leq n_i+k$ for each $i$, so $l(\underline x) \in  \{0,\dots, k-1\}^{mp-1}$. 

Given $\bar l 
\in \{0,\dots,k-1\}^{mp-1}$, let $\Gamma^{\bar l} \subset \prod_{i=1}^m (\Gamma_1 \times \cdots  \times \Gamma_p)$ be the set of all $\underline x$ such that $ l(\underline x) = \bar l$. If $\underline x, \underline x' \in \Gamma^{\bar l}$ and $i \in \{1, \ldots, k-1\}$, then by construction, $\tau_1(\underline x) + \cdots + \tau_i(\underline x)$ and $\tau'_1(\underline x) + \cdots + \tau'_i(\underline x)$ belong to the same element of the partition $P$.

We show that $\Phi_m$ is 1-1 on each $\Gamma^{\bar l}$. Fix $\bar l$ and let $\underline x,\underline x'\in \Gamma^{\bar l}$ be distinct.  Let $j$ be the smallest index such that $x_j\neq x_j'$. Write $\tau_i=\tau_i(\underline x)$ and $\tau_i' = \tau_i(\underline x')$.  Let $r = \sum_{i=1}^j (t_i + \tau_i)$ and $r' = \sum_{i=1}^j (t_i + \tau_i')$.  Since $\sum_{i=1}^j \tau_i$ and $ \sum_{i=1}^j \tau_i$ belong to the same element of $P$, then $|r-r'| = |\sum_{i=1}^j \tau_i - \sum_{i=1}^j \tau_i'|  < \zeta$.

Because $x_j\neq x_j'\in \Gamma_i$ for some $i\in\{i, \ldots, p\}$ and $\Gamma_i$ is $(t_i,3\epsilon)$-separated, we have $d_{t_i}(x_j,x_j') >  3\epsilon$.  Now we have
\[
d_{b_m}(\Phi_m \underline x, \Phi_m \underline x') \geq d_{t_i}(f_{r}\Phi_m \underline x,f_{r}\Phi_m \underline x') > d_{t_i}(f_{r}\Phi_m \underline x, f_{r'}\Phi_m\underline x') - \epsilon/2,
\]
where the $\epsilon/2$ term comes from the fact that $d_{t_i}(f_{r}\Phi_m\underline x',f_{r'}\Phi_m\underline x') \leq \epsilon/2$ by our choice of $\zeta$.  For the first term, observe that
\[
d_{t_i}(f_{r}\Phi_m \underline x,f_{r'}\Phi_m \underline x') \geq d_{t_i}(x_j, x_j') - d_{t_i}(x_j, f_{r}\Phi_m\underline x) - d_{t_i}(f_{r'}\Phi_m\underline x', x_j')> d_{t_i}(x_j, x_j') -2\epsilon.
\]
It follows that $d_{b_m}(\Phi_m\underline x,\Phi_m\underline x') >  \epsilon/2$. Thus, $\Phi_m$ is 1-1 on $\Gamma^{\bar l}$ and $\Phi_m(\Gamma^{\bar l})$ is $(b_m, \epsilon/2)$-separated. There are $k^{mp-1}$ choices for $\bar l$, so letting $C= k^p$, by the pigeon hole principle, there exists $\bar l$ so that $\# \Gamma^{\bar l} \geq C^{-m} \#(\prod_{i=1}^m \Gamma_1 \times \cdots  \times \Gamma_p)$. For this $\bar l$, we let $E_m := \Phi_m(\Gamma^{\bar l})$.
\end{proof}

We have that 
\begin{align*}
C^m \# E_m \geq (\prod_{i=1}^p \# \Gamma_i)^m \geq  e^{m \sum_{i=1}^pt_ih_i} =  e^{m t\sum_{i=1}^pa_ih_i} &\geq   e^{m t\sum_{i=1}^pa_i(h_{\mu_i}-\eta)} = e^{mt(h_\lambda - \eta)}.
\end{align*}
Thus, $\frac{1}{tm} \log \# E_m > h_{\mu}-2\eta - \frac{1}{t} \log C$. Note that $b_m \leq  m \sum_{i=1}^pt_i +mp\tau = m(t+p \tau)$,
and thus $tm/b_m \geq t/(t+p\tau)$. Sending $m \to \infty$, we obtain 
\[
h(Y) \geq \liminf_{m \to \infty} \frac{tm}{b_m} \frac{1}{tm}\log \#E_m \geq \frac{t}{t+p\tau}(h_\mu-2\eta-\frac{1}{t} \log C).
\]
This is true for all large $t$, so this shows that  $h(Y) \geq h_{\mu}- 2\eta$.

Since $h(Y) = \sup \{ h_\nu : \nu \text{ is ergodic and } \nu \in \mathcal M_{\FFF^+}(Y)$\}, we can find an ergodic measure $\nu$ supported on $Y$ with $h_\nu \geq h_\mu - 2 \eta$.
The discussion preceding Lemma \ref{ergodic} shows that $\nu \in \mathcal M_{\FFF}(X)$. Thus $\nu$ satisfies the conclusion of Proposition \ref{entropydenseflow}.

\subsection{Large Deviations Principle}
We obtain the large deviations principle for all the measures considered in this section. The large deviations principle is a statement which describes the decay rate of the measure of points whose Birkhoff sums are experiencing a large deviation from their expected value given by the Birkhoff ergodic theorem.

\begin{defn} \label{def:ldp}
Let $m$ be an equilibrium measure for a potential $\varphi$ (with respect to $\mathcal F$). We say that $m$ satisfies the \emph{upper large deviations principle} if for any continuous observable $\psi\colon X \to \mathbb R$ and any $\epsilon>0$,we have
\begin{equation} \label{ldp}
\limsup_{t \to \infty} \frac 1t \log m\left \{x : \left |\frac 1t \int_0^t \psi(f_sx)\,ds - \int \psi\, dm \right|\geq \epsilon \right \} \leq -q(\epsilon),
\end{equation}
where the rate function $q$ is given by
\begin{equation}\label{eqn:rate-function}
q(\epsilon) := P(\varphi) - \sup_{\substack{\nu\in\mathcal{M}_{\mathcal F}(X) \\ \left\lvert \int \psi\,dm - \int \psi\,d\nu\right\rvert \geq \epsilon }}
\left(h_{\nu} (f ) + \int \varphi \, d \nu \right),
\end{equation}
or $q(\epsilon)= \infty$ when $\{\nu\in\mathcal{M}_{\mathcal F}(X) : \left\lvert \int \psi\,dm - \int \psi\,d\nu\right\rvert \geq \epsilon\} = \emptyset$. We say that the \emph{lower large deviations principle} holds if the above statement holds with $\geq$ in place of $\leq$, and $\liminf$ in place of $\limsup$  in \eqref{ldp}. We say that $m$ satisfies the \emph{large deviations principle} if both upper and lower large deviations hold: that is, the above statement holds with equality in place of $\leq$ in \eqref{ldp}, and the $\limsup$ becomes a limit. For a discrete-time dynamical system $(X,f)$, we say the \emph{lower large deviations principle} holds (and similarly for \emph{upper}) if the above statement holds with $t$ replaced by $n$ and $\frac 1t \int_0^t \varphi(f_sx)\,ds$ replaced by $\sum_{i=0}^{n-1} \varphi(f^i x)$ in \eqref{ldp}, and $\mathcal{M}_{\mathcal F}(X)$ replaced by $\mathcal{M}_{f}(X)$ in \eqref{eqn:rate-function}.
\end{defn}

For a fixed observable $\psi$, the statement above is known as the level-1 large deviations principle.   If level-1 large deviations holds for \emph{every} continuous observable $\psi$ (as opposed to, say, only for every H\"older continuous or smooth $\psi$), then this is equivalent to the level-2 large deviations principle.  The level-2 property is often formulated as a large deviation result for empirical measures, i.e. a description of the rate of decay of the measure of the set of points $x$ satisfying $D(\mathcal E_t(x), m)\geq \epsilon$ as $t\to \infty$.
 See \cite{CRL, Yamamoto} for a precise statement of this formulation, and the argument that level-2 large deviations follows from the statement of Definition \ref{def:ldp}. We have the following result.

\begin{prop}\label{prop:large deviations}
For an expansive flow $(X, \FFF)$ with weak specification and a continuous function $\varphi:X \to \R$ with the Bowen property,  the unique equilibrium state satisfies the large deviations principle. 
\end{prop}

A large deviations result for measures with a weak Gibbs property for semi-flows  (i.e. continuous systems $(X, \{f_t\}_{t\geq0})$ which may not be invertible) with weak specification was announced in the preprint \cite{BV}. Since every flow is a semi-flow, and our equilibrium states have the Gibbs property, those results apply here.  We give a short independent proof using the entropy density of ergodic measures, which is not proved in \cite{BV}. We treat the upper and lower large deviations bounds separately.


\subsubsection{Upper large deviations}
For the upper large deviations principle, we can reduce to considering the time-$1$ map of the flow. It is easy to see that the upper large deviations principle for the flow follows from the upper large deviations principle for the time-$1$ map. This follows because \eqref{ldp} can be verified for any continuous function $\psi$ by applying the large deviations principle for the time-1 map to the continuous function $\psi_1:=\int_0^1 \psi(f_sx)ds$.

The Gibbs property \eqref{eqn:gibbs} for the flow immediately yields the Gibbs property with respect to the time-$1$ map.
\[
Q^{-1} e^{-tP(\varphi) + \sum_{i=0}^{n-1} \varphi_1(f^i x)} \leq \mu(B_n(x,\rho; f_1)) \leq Q e^{-tP(\varphi) + \sum_{i=0}^{n-1} \varphi_1(f^i x)},
\]
where $B_n(x, \epsilon; f_1) = \{y : d_1(f_1^ix, f_1^iy)< \rho \text{ for all } i \in \{0, \ldots, n-1\} \}$, and $d_1$ is the metric equivalent to $d$ given by $d_1(x, y) = \sup_{t\in[0,1)} d(f_tx, f_ty)$. Note also that from the variational principle and flow invariance of the measure $P(\varphi_1, f_1)=P(\varphi, \mathcal F)$.

It is well known that in the discrete-time case the upper large deviations principle follows from the upper Gibbs property and upper semi-continuity of the entropy map $\mu \to h_\mu$ (which follows from expansivity of the flow). This follows from Theorem 3.2 of \cite{PfS}, whose hypotheses are the existence of an \emph{upper-energy function} and upper semi-continuity of the entropy map. The existence of an upper-energy function $e_\mu$ can easily be deduced from the upper bound in the Gibbs property and by setting $e_\mu:= P(\phi_1, f_1)- \phi_1(x)$. See  \S7.2 of \cite{ClimenhagaThompsonflows} for this argument.

Thus, we have the upper large deviations for $\varphi_1$ for $\mu$ with respect to $f_1$, and thus the upper large deviations principle for $\varphi$ with respect to the flow of \eqref{ldp}. 


\subsubsection{Lower large deviations}

We now verify the lower large deviations principle. In the discrete-time case, lower large deviations is proved as Theorem 3.1 of Pfister and Sullivan \cite{PfS} under the following three hypotheses (see also Theorem 3.1 of \cite{Yamamoto}):
\begin{enumerate}
\item Upper semi-continuity of the entropy map;

\item Existence of a ``lower-energy function'', which follows easily from the lower Gibbs property;

\item Entropy density of ergodic measures in the space of invariant measures.
\end{enumerate}
The entropy density of ergodic measures is the most difficult hypothesis to check, and we carried this out in \S\ref{sec:hdense}. The rest of the argument is fairly standard. Nevertheless, we do not know of a reference in continuous time, so we sketch the proof. First observe that it is clear that entropy density of ergodic measures means that it is possible to consider only ergodic measures in the expression
\[
\sup \left \{ h_{\nu} (f ) + \int \varphi d \nu : \left\lvert \int \psi\,dm - \int \psi\,d\nu\right\rvert \geq \epsilon  \right \}.
 \]
Thus, for the lower large deviations, it will suffice  to show that for any ergodic $\mu$ with $\left\lvert \int \psi\,dm - \int \psi\,d\nu\right\rvert > \epsilon$  and $\delta>0$ sufficiently small that
\begin{equation} \label{bound}
\lim_{t \to \infty} \frac 1t \log m\left \{x : \left |\frac 1t \int_0^t \psi(f_sx)\,ds - \int \psi\, d\mu \right|\leq \delta \right \} \geq P(\varphi) - (h_{\mu} + \int \varphi d \mu).
\end{equation}
This is achieved by a combination of the Gibbs property for $m$, and basic cardinality estimates for $\mu$. A sketch goes as follows. For a suitable small $\eta>0$, from the Katok entropy formula, and the Birkhoff ergodic theorem, we can find a sequence of $(t, \eta)$ separated sets with $\#E_t > e^{t(h_{\mu}-\eta)}$ so that for $\chi \in\{\varphi, \psi\}$, we have
\[
\sup_{y \in B_t(x, \eta), x \in E_t} \left |\frac 1t \int_0^t \chi(f_sx)\,ds - \int \chi\, d\mu \right|\leq \delta. 
\]
 Then 
\[
m\left \{x : \left |\frac 1t \int_0^t \psi(f_sx)\,ds - \int \psi\, d\mu \right|\leq \delta \right \} \geq \sum_{x \in E_t} m(B_t(x, \eta)).
\]
 By the Gibbs property, $m(B_t(x, \eta)) \geq C^{-1} e^{-tP(\varphi)+ \int_0^t \varphi(f_sx)ds}$, and since $x \in E_t$, $\int_0^t \psi(f_sx)ds \geq \int t\psi d \mu-t\delta$.
 Thus
\begin{align*}
m\left \{x : \left |\frac 1t \int_0^t \psi(f_sx)\,ds - \int \psi\, d\mu \right|\leq \delta \right \} & \geq Q^{-1} \# E_t e^{-tP(\varphi)+ t\psi d \mu-t\delta} \\ &\geq Q^{-1} e^{-t(P(\varphi)- (h_{\mu} + \int \varphi d \mu)+ \eta + \delta)}
\end{align*}
The proof of the lower large deviations principle follows.

%

\bibliographystyle{amsplain}
\bibliography{biblio}

\providecommand{\bysame}{\leavevmode\hbox to3em{\hrulefill}\thinspace}
\providecommand{\MR}{\relax\ifhmode\unskip\space\fi MR }
\providecommand{\MRhref}[2]{%
  \href{http://www.ams.org/mathscinet-getitem?mr=#1}{#2}
}
\providecommand{\href}[2]{#2}
\begin{thebibliography}{10}

\bibitem{ballmann}
Werner Ballmann, \emph{Lectures on spaces of nonpositive curvature}, DMV
  Seminar, vol.~25, Birkh\"auser, Basel, 1995.

\bibitem{BV}
Thiago Bomfim and Paulo Varandas, \emph{The gluing orbit property, uniform
  hyperbolicity and large deviations principles for semiflows}, Preprint,
  ArXiv:1507.03905, to appear in Journal of Differential Equations, 2019.

\bibitem{BW}
R.~Bowen and P.~Walters, \emph{Expansive one-parameter flows}, J. Differential
  Equations \textbf{12} (1972), 180--193.

\bibitem{bowen-periodic}
Rufus Bowen, \emph{Periodic orbits for hyperbolic flows}, American Journal of
  Mathematics \textbf{94} (1972), no.~1, 1--30.

\bibitem{rB73}
Rufus Bowen, \emph{Symbolic dynamics for hyperbolic flows}, Amer. J. Math.
  \textbf{95} (1973), 429--460.

\bibitem{bowen}
Rufus Bowen, \emph{Some systems with unique equilibrium states}, Mathematical
  Systems Theory \textbf{8} (1975), no.~3, 193--202.

\bibitem{Bowen-Ruelle}
Rufus Bowen and David Ruelle, \emph{The ergodic theory of {A}xiom {A} flows},
  Inventiones Mathematicae \textbf{29} (1975), no.~3, 181--202.

\bibitem{BPP16}
A.~Broise-Alamichel, J.~Parkkonen, and F.~Paulin, \emph{Equidistribution and
  counting under equilibrium states in negative curvature and trees.
  applications to non-archimedean diophantine approximation}, Progress in
  Mathematics, vol. 329, Birkhauser, 2019, To appear.

\bibitem{BCFT}
K.~Burns, V.~Climenhaga, T.~Fisher, and D.~J. Thompson, \emph{Unique
  equilibrium states for geodesic flows in nonpositive curvature}, Geom. Funct.
  Anal. \textbf{28} (2018), no.~5, 1209--1259.

\bibitem{champetier}
Christophe Champetier, \emph{Petite simplification dans les groupes
  hyperboliques}, Ann. Fac. Sci. Toulouse Math. \textbf{3} (1994), no.~2,
  161--221.

\bibitem{ClimenhagaThompsonflows}
Vaughn Climenhaga and Daniel~J. Thompson, \emph{Unique equilibrium states for
  flows and homeomorphisms with non-uniform structure}, Adv. Math. \textbf{303}
  (2016), 744--799.

\bibitem{CRL}
Henri Comman and Juan Rivera-Letelier, \emph{Large deviation principles for
  non-uniformly hyperbolic rational maps}, Ergodic Theory Dynam. Systems
  \textbf{31} (2011), no.~2, 321--349.

\bibitem{CLMT}
D.~Constantine, J-F. Lafont, D.~McReynolds, and D.J. Thompson, \emph{Fat flats
  for rank one manifolds}, Preprint, arXiv:1704.00857, to appear in Michigan
  Mathematical Journal, 2019.

\bibitem{cp}
Michel Coornaert and Athanase Papadopoulos, \emph{Symbolic coding for the
  geodesic flow associated to a word hyperbolic group}, Manuscripta Math.
  \textbf{109} (2012), 465--492.

\bibitem{CS2}
Yves Coudene and Barbara Schapira, \emph{Generic measures for hyperbolic flows
  on non-compact spaces}, Israel Journal of Mathematics \textbf{179} (2010),
  no.~1, 157--172.

\bibitem{CS}
Yves Coud\`{e}ne and Barbara Schapira, \emph{Counterexamples in non-positive
  curvature}, Discrete \& Continuous Dynamical Systems-A \textbf{30} (2011),
  no.~4, 1095--1106.

\bibitem{De92}
Manfred Denker, \emph{Large deviation and the pressure function}, Transactions
  of the 11th Prague Conference on Information theory, Statistical Decision
  Functions, Prague, 1990, Springer, Berlin, 1992, pp.~21--33.

\bibitem{DGS}
Manfred Denker, Christian Grillenberger, and Karl Sigmund, \emph{Ergodic theory
  on compact spaces}, Lecture Notes in Mathematics, vol. 527, Springer-Verlag,
  Berlin-New York, 1976.

\bibitem{EKW}
A.~Eizenberg, Y.~Kifer, and B.~Weiss, \emph{Large deviations for
  $\mathbb{Z}^d$-actions}, Comm. Math. Phys \textbf{164} (1994), no.~3,
  433--454.

\bibitem{Ellis}
Richard~S. Ellis, \emph{Entropy, large deviations, and statistical mechanics},
  Grundlehren der Mathematischen Wissenschaften [Fundamental Principles of
  Mathematical Sciences], vol. 271, Springer-Verlag, New York, 1985.

\bibitem{eF77}
Ernesto Franco, \emph{Flows with unique equilibrium states}, Amer. J. Math.
  \textbf{99} (1977), no.~3, 486--514.

\bibitem{GK}
Katrin Gelfert and Dominik Kwietniak, \emph{On density of ergodic measures and
  generic points}, Ergodic Theory and Dynamical Systems \textbf{38} (2018),
  no.~5, 1745--1767.

\bibitem{GM}
Katrin Gelfert and Adilson~E. Motter, \emph{({N}on)invariance of dynamical
  quantities for orbit equivalent flows}, Comm. Math. Phys. \textbf{300}
  (2010), no.~2, 411--433.

\bibitem{GS14}
Katrin Gelfert and Barbara Schapira, \emph{Pressures for geodesic flows of rank
  one manifolds}, Nonlinearity \textbf{27} (2014), no.~7, 1575--1594.

\bibitem{GP}
Anton Gorodetski and Yakov Pesin, \emph{Path connectedness and entropy density
  of the space of hyperbolic ergodic measures}, Modern Theory of Dynamical
  Systems: A Tribute to Dmitry Victorovich Anosov \textbf{692} (2017), 111.

\bibitem{gromov}
M.~Gromov, \emph{Hyperbolic groups}, Essays in Group Theory (S.~Gersten, ed.),
  MSRI Publications, vol.~8, Springer, 1987, pp.~75--265.

\bibitem{vK90}
Vadim~A. Kaimanovich, \emph{Invariant measures of the geodesic flow and
  measures at infinity on negatively curved manifolds}, Ann. Inst. H.
  Poincar\'e Phys. Th\'eor. \textbf{53} (1990), no.~4, 361--393, Hyperbolic
  behaviour of dynamical systems (Paris, 1990).

\bibitem{vK91}
\bysame, \emph{Bowen-{M}argulis and {P}atterson measures on negatively curved
  compact manifolds}, Dynamical systems and related topics ({N}agoya, 1990),
  Adv. Ser. Dynam. Systems, vol.~9, World Sci. Publ., River Edge, NJ, 1991,
  pp.~223--232.

\bibitem{KH}
Anatole Katok and Boris Hasselblatt, \emph{Introduction to the modern theory of
  dynamical systems}, Cambridge University Press, 1995.

\bibitem{knieper1}
Gerhard Knieper, \emph{The uniqueness of the measure of maximal entropy for
  geodesic flows on rank {$1$} manifolds}, Ann. of Math. \textbf{148} (1998),
  no.~1, 291--314.

\bibitem{knieper2}
\bysame, \emph{The uniqueness of the maximal measure for geodesic flows on
  symmetric spaces of higher rank}, Israel J. Math. \textbf{149} (2005),
  171--183.

\bibitem{KT17}
T.~Kucherenko and D.J. Thompson, \emph{Measures of maximal entropy for
  suspension flows over the full shift}, to appear in Mathematische
  Zeitschrift, 2019.

\bibitem{LL10}
Fran{\c{c}}ois Ledrappier and Seonhee Lim, \emph{Volume entropy of hyperbolic
  buildings}, J. Mod. Dyn. \textbf{4} (2010), no.~1, 139--165.

\bibitem{LS}
Yuri Lima and Omri~M. Sarig, \emph{Symbolic dynamics for three-dimensional
  flows with positive topological entropy}, J. Eur. Math. Soc. (JEMS)
  \textbf{21} (2019), no.~1, 199--256.

\bibitem{LM}
D.~Lind and B.~Marcus, \emph{An introduction to symbolic dynamics and coding},
  Cambridge University Press, 1995.

\bibitem{matheus}
Fr\'ed\'eric Math\'eus, \emph{Flot g\'eod\'esique et groupes hyperboliques
  d'apr\'es {M}. {G}romov}, S\'eminaire de Th\'eorie Spectrale et
  G\'{e}om\'{e}trie, no.~9, 1991, pp.~67--87.

\bibitem{N}
Sheldon~E. Newhouse, \emph{Lectures on dynamical systems}, Progr. Math.,
  vol.~8, Birkh\"auser, Boston, Mass., 1980.

\bibitem{OP}
Steven Orey and Stephan Pelikan, \emph{Large deviation principles for
  stationary processes}, Ann. Probab. \textbf{16} (1988), no.~4, 1481--1495.

\bibitem{PP}
W.~Parry and M.~Pollicott, \emph{Zeta functions and the periodic orbit
  structure of hyperbolic dynamics}, Ast\'erisque, no. 187-188, Soc. Math.
  France, 1990.

\bibitem{wP88}
William Parry, \emph{Equilibrium states and weighted uniform distribution of
  closed orbits}, Dynamical systems ({C}ollege {P}ark, {MD}, 1986--87), Lecture
  Notes in Math., vol. 1342, Springer, Berlin, 1988, pp.~617--625.

\bibitem{PPS}
Fr{\'e}d{\'e}ric Paulin, Mark Pollicott, and Barbara Schapira,
  \emph{Equilibrium states in negative curvature}, Ast\'erisque (2015),
  no.~373, viii+281.

\bibitem{PfS}
C-E Pfister and W~G Sullivan, \emph{Large deviations estimates for dynamical
  systems without the specification property. {A}pplication to the
  $\beta$-shifts}, Nonlinearity \textbf{18} (2005), 237--261.

\bibitem{pollicott_nonpos}
Mark Pollicott, \emph{Closed geodesic distribution for manifolds of
  non-positive curvature}, Discrete Contin. Dyn. Syst. \textbf{2} (1996),
  no.~2, 153--161.

\bibitem{roblin}
Thomas Roblin, \emph{Ergodicit\'e et \'equidistribution en courbure
  n\'egative}, M\'em. Soc. Math. Fr. (N.S.) (2003), no.~95, vi+96.

\bibitem{waddington}
Simon Waddington, \emph{Large deviation asymptotics for {A}nosov flows}, Ann.
  Inst. H. Poincar\'e Anal. Non Lin\'eaire \textbf{13} (1996), no.~4, 445--484.

\bibitem{Walters}
Peter Walters, \emph{An introduction to ergodic theory}, Graduate Texts in
  Mathematics, vol.~79, Springer-Verlag, New York-Berlin, 1982.

\bibitem{Yamamoto}
Kenichiro Yamamoto, \emph{On the weaker forms of the specification property and
  their applications}, Proceedings of the AMS \textbf{137} (2009), no.~11,
  3807--3814.

\bibitem{Yo}
L.S. Young, \emph{Large deviations in dynamical systems}, Trans. Amer. Math.
  Soc. \textbf{318} (1990), no.~2, 525--543.

\end{thebibliography}

\end{document}